\documentclass[11pt,leqno]{article}

\usepackage[active]{srcltx}

\usepackage{amsfonts,latexsym,amsmath,amssymb,amsthm}
\usepackage{fullpage}
\newtheorem{theorem}{Theorem}[section]
\newtheorem{lemma}[theorem]{Lemma}

\newtheorem{proposition}[theorem]{Proposition}
\newtheorem{corollary}[theorem]{Corollary}
\newtheorem{remark}[theorem]{Remark}

\theoremstyle{definition}
\newtheorem{definition}[theorem]{Definition}

\theoremstyle{remark}

\newtheorem*{note*}{Note}

\numberwithin{equation}{section}

\makeatletter

\newcommand{\rank}{\mathop{\operator@font rank}}
\newcommand{\conv}{\mathop{\operator@font conv}}
\newcommand{\vol}{\mathrm{vol}}

\newcommand{\onetagright}{\tagsleft@false}
\makeatother

\newcommand{\ls}{\leqslant}
\newcommand{\gr}{\geqslant}

\usepackage{color}

\usepackage{hyperref}

\begin{document}
\small

\title{\bf Norms of weighted sums of log-concave random vectors}

\medskip

\author{Giorgos Chasapis, Apostolos Giannopoulos and Nikos Skarmogiannis}

\date{}

\maketitle

\begin{abstract}
\footnotesize \noindent Let $C$ and $K$ be centrally symmetric convex bodies of volume $1$ in ${\mathbb R}^n$.
We provide upper bounds for the multi-integral expression
\begin{equation*}\|{\bf t}\|_{C^s,K}=\int_{C}\cdots\int_{C}\Big\|\sum_{j=1}^st_jx_j\Big\|_K\,dx_1\cdots dx_s\end{equation*}
in the case where $C$ is isotropic. Our approach provides an alternative proof of the sharp lower bound, due to Gluskin and V.~Milman,
for this quantity. We also present some applications to ``randomized" vector balancing problems.
\end{abstract}

\section{Introduction}

Let $K$ be a centrally symmetric convex body in ${\mathbb R}^n$. For any $s$-tuple ${\cal C}=(C_1,\ldots ,C_s)$ of centrally symmetric convex bodies $C_j$
in ${\mathbb R}^n$ we consider the norm on ${\mathbb R}^s$, defined by
\begin{equation*}\|{\bf t}\|_{{\cal C},K}=\frac{1}{\prod_{j=1}^s {\rm vol}_n(C_j)}\int_{C_1}\cdots\int_{C_s}\Big\|\sum_{j=1}^st_jx_j\Big\|_K\,dx_1\cdots dx_s,\end{equation*}
where ${\bf t}=(t_1,\ldots ,t_s)$. If ${\cal C}=(C,\ldots ,C)$ then we write $\|{\bf t}\|_{C^s,K}$ instead of $\|{\bf t}\|_{{\cal C},K}$.
A question posed by V.~Milman is to determine if, in the case $C=K$, one has that $\|\cdot\|_{K^s,K}$ is equivalent to the
standard Euclidean norm up to a term which is logarithmic in the dimension, and in particular, if under some cotype condition
on the norm induced by $K$ to ${\mathbb R}^n$ one has equivalence between $\|\cdot\|_{K^s,K}$ and the Euclidean norm.

This question was studied by Bourgain, Meyer, V.~Milman and Pajor in \cite{BMMP}; they obtained the lower bound
\begin{equation*}\|{\bf t}\|_{{\cal C},K}\gr c\sqrt{s}\Big (\prod_{j=1}^s|t_j|\Big)^{1/s}\Big (\prod_{j=1}^s{\rm vol}_n(C_j)\Big)^{\frac{1}{sn}}/{\rm vol}_n(K)^{1/n},\end{equation*}
where $c>0$ is an absolute constant. Gluskin and V.~Milman studied the same question in \cite{Gluskin-VMilman-2004} and obtained
a better lower bound in a more general context.

\begin{theorem}[Gluskin-Milman]\label{th:extremal-balls-1} Let $A_1,\ldots ,A_s$ be measurable sets in ${\mathbb R}^n$ and $K$ be a star body in
${\mathbb R}^n$ with $0\in {\rm int}(K)$. Then, for all ${\bf t}=(t_1,\ldots ,t_s)\in {\mathbb R}^s$,
\begin{equation*}\|{\bf t}\|_{{\cal A},K}:=\frac{1}{\prod_{j=1}^s {\rm vol}_n(A_j)}\int_{A_1}\cdots\int_{A_s}\Big\|\sum_{j=1}^st_jx_j\Big\|_K\,dx_1\cdots dx_s \gr c\Big(\sum_{j=1}^st_j^2\Big (\frac{{\rm vol}_n(A_j)}{{\rm vol}_n(K)}\Big )^{2/n}\Big )^{1/2},\end{equation*}
where $c>0$ is an absolute constant. Equivalently, if ${\rm vol}_n(A_j)={\rm vol}_n(K)$ for all $1\ls j\ls s$ then
\begin{equation}\label{eq:GM-basic}\|{\bf t}\|_{{\cal A},K}\gr c\,\|{\bf t}\|_2\end{equation} for all ${\bf t}\in {\mathbb R}^s$.
\end{theorem}

In the statement above, when $K$ is a star body with respect to $0$ we use the notation $\|x\|_K$ for the gauge function of $K$, defined
by $\inf\{r>0: x/r\in K\}$. The proof of Theorem~\ref{th:extremal-balls-1} actually shows that one can have $c\gr
c(n)/\sqrt{2}$, where $c(n)\to 1$ as $n\to\infty $. Gluskin and V.~Milman use a symmetrization type result which is a consequence of the
Brascamp-Lieb-Luttinger inequality: under the assumptions of Theorem~\ref{th:extremal-balls-1} and the additional assumption that
${\rm vol}_n(A_j)={\rm vol}_n(K)={\rm vol}_n(B_2^n)$ for all $1\ls j\ls s$, one has
\begin{align*}
& {\rm vol}_{ns}\Big ( \Big\{(x_j)_{1\ls j\ls s}: x_j\in A_j\;\;\hbox{for
all}\;j\;\;\hbox{and}\;\Big\|\sum_{j=1}^st_jx_j\Big\|_K<\alpha \Big\}\Big)\\
& \hspace*{2cm}\ls {\rm vol}_{ns}\Big (\Big\{ (x_j)_{1\ls j\ls s}: x_j\in
B_2^n\;\;\hbox{for
all}\;j\;\;\hbox{and}\;\Big\|\sum_{j=1}^st_jx_j\Big\|_2<\alpha \Big\}\Big)
\end{align*}for any ${\bf t}=(t_1,\ldots ,t_s)\in {\mathbb R}^s$ and
any $\alpha >0$.

Our starting point is a simple but useful identity; one has
\begin{equation}\label{eq:first-identity}\|{\bf t}\|_{{\cal C},K}=\|{\bf t}\|_2\,\int_{{\mathbb R}^n}\|x\|_K\,d\nu_{{\bf t}}(x),\end{equation}
where $\nu_{{\bf t}}$ is the distribution of the random vector $\frac{1}{\|{\bf t}\|_2}(t_1X_1+\cdots +t_sX_s)$ and $X_j$ are
independent random vectors uniformly distributed on $C_j$. Starting with \eqref{eq:first-identity} we can actually give an
alternative short proof of Theorem~\ref{th:extremal-balls-1} in the case that we study.

\begin{theorem}\label{th:GM-short}
Let ${\cal C}=(C_1,\ldots ,C_s)$ be an $s$-tuple of centrally symmetric convex bodies and $K$ be a centrally symmetric convex body in $\mathbb{R}^n$
with ${\rm vol}_n(C_j)={\rm vol}_n(K)=1$. Then, for any ${\bf t}=(t_1,\ldots,t_s)\in \mathbb{R}^s$,
\begin{equation*}\|{\bf t}\|_{{\cal C},K}\gr \frac{n}{e(n+1)}\|{\bf t}\|_2.\end{equation*}
\end{theorem}

We are mainly interested in upper bounds for the quantity $\|{\bf t}\|_{C^s,K}$. Since $\|{\bf t}\|_{C^s,K}=\|{\bf t}\|_{(TC)^s,TK}$
for any $T\in SL(n)$, we may restrict our attention to the case where $C$ is isotropic
(see Section~\ref{sec:background} for the definition and background information). In this case
\begin{equation}\label{eq:basic-identity}\|{\bf t}\|_{C^s,K}=\|{\bf t}\|_2L_C\, I_1(\mu_{{\bf t}},K),\end{equation}
where $\mu_{{\bf t}}$ is an isotropic, compactly supported log-concave probability measure depending on ${\bf t}$ and,
for any centered log-concave probability measure $\mu $ on ${\mathbb R}^n$,
\begin{equation*}I_1(\mu ,K)=\int_{{\mathbb R}^n}\|x\|_Kd\mu (x).\end{equation*}
In order to get a feeling of what one would expect, let us note that if $\mu $ is an isotropic log-concave probability measure
on ${\mathbb R}^n$ and $K$ is a centrally symmetric convex body of volume $1$ in ${\mathbb R}^n$ then
\begin{equation*}\int_{O(n)}I_1(\mu ,U(K))\,d\nu (U) =\int_{{\mathbb R}^n}\int_{O(n)}\|x\|_{U(K)}d\nu (U)\,d\mu (x)=
M(K)\int_{{\mathbb R}^n}\|x\|_2d\mu (x)\approx \sqrt{n}M(K),\end{equation*}
where
\begin{equation*}M(K):=\int_{S^{n-1}}\|\xi\|_Kd\sigma (\xi)\end{equation*}
and $\nu ,\sigma $ denote the Haar probability measures on $O(n)$ and $S^{n-1}$ respectively. It follows that
\begin{equation}\label{eq:intro-1}\int_{O(n)}\|{\bf t}\|_{U(C)^s,K}\approx (L_C\,\sqrt{n}M(K))\,\|{\bf t}\|_2.\end{equation}
Therefore, our goal is to obtain a constant of the order of $L_C\,\sqrt{n}M(K)$ in our upper estimate for $\|{\bf t}\|_{C^s,K}$.
Let us note here that the question to estimate the parameter $M(K)$ for an isotropic centrally symmetric convex body $K$ in ${\mathbb R}^n$,
which will appear frequently in our upper bounds, remains open; one may hope that
$L_K\,\sqrt{n}M(K)\ls c(\log n)^b$ for some absolute constant $b>0$. However, the currently best known estimate is
\begin{equation*}M(K)\ls \frac{c\log^{2/5}(e+n)}{\sqrt[10]{n}L_K}.\end{equation*}
This is proved in \cite{Giannopoulos-EMilman-2014} (see  also \cite{GSTV} for previous work on this question) and it is also shown
that in the case where $K$ is a $\psi_2$-body with constant $\varrho $ one has
\begin{equation*}M(K)\ls \frac{c\sqrt[3]{\varrho }\log^{1/3}(e+n)}{\sqrt[6]{n}L_K}.\end{equation*}
We pass now to our bounds for $\|{\bf t}\|_{C^s,K}$. Some straightforward upper and lower estimates are given in the next theorem.

\begin{theorem}\label{th:first}
Let $C$ be an isotropic convex body in $\mathbb{R}^n$ and $K$ be a centrally symmetric convex body in $\mathbb{R}^n$. Then, for any $s\gr 1$
and ${\bf t}=(t_1,\ldots,t_s)\in \mathbb{R}^s$,
\begin{equation*}c_1L_CR(K^{\circ })\, \|{\bf t}\|_2\ls \|{\bf t}\|_{C^s,K}\ls \sqrt{n}L_CR(K^{\circ })\, \|{\bf t}\|_2,\end{equation*}
where $c_1>0$ is an absolute constant and $R(K^{\circ })$ is the radius of $K^{\circ }$.
\end{theorem}

A class of centrally symmetric convex bodies for which the upper bound of Theorem~\ref{th:first} can be applied is the class of $2$-convex
bodies. More precisely, in Section~4.1 we see that if $K$ is an isotropic convex body in ${\mathbb R}^n$, which is also $2$-convex
with constant $\alpha$, then
\begin{equation*}\|{\bf t}\|_{C^s,K} \ls (c_2L_C/\sqrt{\alpha })\, \|{\bf t}\|_2\end{equation*}
for any isotropic centrally symmetric convex body $C$ and any ${\bf t}=(t_1,\ldots,t_s)\in \mathbb{R}^s$, where $c_2>0$ is an absolute constant.
In particular, for any centrally symmetric convex body $K$ in ${\mathbb R}^n$ which is $2$-convex with constant $\alpha $  we have
\begin{equation*}\|{\bf t}\|_{K^s,K}  \ls (c_3/\alpha )\, \|{\bf t}\|_2\end{equation*}
for all ${\bf t}=(t_1,\ldots,t_s)\in \mathbb{R}^s$, where $c_3>0$ is an absolute constant.

\smallskip

Starting again with \eqref{eq:basic-identity} and using an argument which goes back to Bourgain (also, employing Paouris' inequality
and Talagrand's comparison theorem) in Section~4.2 we obtain a general upper bound of different type.

\begin{theorem}\label{th:general}
Let $C$ be an isotropic convex body in $\mathbb{R}^n$ and $K$ be a centrally symmetric convex body in $\mathbb{R}^n$. Then,
\begin{equation*}
\|{\bf t}\|_{C^s,K}\ls c\,\Big (L_C\max\Big\{ \sqrt[4]{n},\sqrt{\log (1+s)}\Big\}\Big )\,\sqrt{n}M(K)\|{\bf t}\|_2
\end{equation*}
for every ${\bf t}=(t_1,\ldots,t_s)\in \mathbb{R}^s$, where $c>0$ is an absolute constant.
\end{theorem}

In the case where $C$ is a $\psi_2$-body with constant $\varrho $, a direct application of Talagrand's theorem leads to a stronger estimate:
If $C$ is an isotropic $\psi_2$-body with constant $\varrho $ and $K$ is a centrally symmetric convex body in $\mathbb{R}^n$ then
\begin{equation*}\|{\bf t}\|_{C^s,K}\ls c\varrho^2\sqrt{n}M(K)\,\|{\bf t}\|_2\end{equation*}
for every ${\bf t}=(t_1,\ldots,t_s)\in \mathbb{R}^s$, where $c>0$ is an absolute constant.

\medskip

Next, combining \eqref{eq:basic-identity} with results of E.~Milman from \cite{EMilman-2006}, we obtain some rather
strong estimates in the case where $K$ has bounded cotype-$2$ constant (see Section~5). In the
case $C=K$ we get:

\begin{theorem}\label{th:intro-type-cotype}Let $K$ be a centrally symmetric convex body in ${\mathbb R}^n$. For any $s\gr 1$ and
any ${\bf t}=(t_1,\ldots ,t_s)\in {\mathbb R}^s$ we have that
\begin{equation*}
\frac{c_3}{C_2(X_K)}\, \|{\bf t}\|_2\ls \|{\bf t}\|_{K^s,K} \ls \big (c_4L_KC_2(X_K)\sqrt{n}M(K_{{\rm iso}})\big)\,\|{\bf t}\|_2,
\end{equation*}
where $C_2(X_K)$ is the cotype-$2$ constant of the normed space $X_K$ with unit ball $K$, and $K_{{\rm iso}}$ is an isotropic image of $K$.
\end{theorem}

In Section~6 we consider the unconditional case; using an argument from \cite{GHT} which is based on well-known results of Bobkov and Nazarov
one has the following estimates.

\begin{theorem}If $K$ and $C_1,\ldots,C_s$ are isotropic unconditional convex bodies
in $\mathbb{R}^n$ then,
\begin{equation*}\|{\bf t}\|_{{\cal C},K}\ls c\sqrt{\log n}\cdot \max\{\|{\bf t}\|_2,\sqrt{\log n}\|{\bf t}\|_\infty\}\end{equation*}
for every ${\bf t}=(t_1,\ldots,t_s)\in \mathbb{R}^s$, where $c>0$ is an absolute constant.
\end{theorem}

As an application of Theorem~\ref{th:intro-type-cotype} and of the ``$\psi_2$-version" of Theorem~\ref{th:general} we can check
that in the special case of the unit ball $B_p^n$ of $\ell_p^n$, $1\ls p\ls\infty $, one has the upper bound
\begin{equation*}\|{\bf t}\|_{\overline{B_p^n}^s,\overline{B_p^n}} \ls c\,\min\{\sqrt{p},\sqrt{\log n}\}\,\|{\bf t}\|_2\end{equation*}
for every $s\gr 1$ and ${\bf t}\in {\mathbb R}^s$, where $c>0$ is an absolute constant (and, generally, $\overline{K}=\vol_n(K)^{-1/n}K$).

In Section~7 we discuss applications of the previous results to some randomized versions
of vector balancing problems. Given two centrally symmetric convex bodies $C, K$ in $\mathbb{R}^n$, the parameter $\beta_s(C,K)$ is
defined as follows:
\begin{equation*}\beta_s(C,K):=\min\Big\{r>0:\;\;\hbox{for any}\;x_1,\ldots,x_s \in C,\;
\min_{\epsilon\in E_2^s}\Big\|\sum_{j=1}^s\epsilon_j x_j \Big\|_K \ls r\Big\},\end{equation*}
where $E_2^s:=\{-1,1\}^s$ is the discrete cube in ${\mathbb R}^s$. Given $x_1,\ldots,x_n \in K$, by the triangle inequality it
is clear that $\|\sum_{j=1}^n \epsilon_j x_j\|_K \ls n$ holds for every $\epsilon\in E_2^n$, thus $\beta_n(K,K)\ls n$.
This bound is actually sharp: taking $K=B_1^n$ and $x_j=e_j$, the standard basis of $\mathbb{R}^n$, we get $\|\sum_{j=1}^n \epsilon_j e_j \|_1 = n$
for every choice of signs. However, the upper bound for $\beta_n(K,K)$ can be significantly better for certain convex bodies,
as suggested for example by a theorem of Spencer \cite{Spe}: one has $\beta_n(B_\infty^n, B_\infty^n)\ls 6\sqrt{n}$.

We further define $\tilde{\beta}(C,K)=\sup_{k\gr n}\beta_k(C,K)$.
Clearly, $\beta_n(C,K)\ls \tilde{\beta}(C,K)$. By a theorem of B\'{a}r\'{a}ny and Grinberg \cite{Bar.Gri},
one has $\tilde{\beta}(K,K)\ls 2n$. This result can also be derived by the trivial bound on $\beta_n(K,K)$ mentioned earlier
and the general observation that
\begin{equation*}\tilde{\beta }(C,K)\ls 2\,\max_{k\ls n}\beta_k(C,K).\end{equation*}
A related result is the Dvoretzky-Hanani lemma (see for example \cite[Lemma 2.2.1]{Kadets.book})
which asserts that for every centrally symmetric convex body $K$ in $\mathbb{R}^n$, for any $s\gr 1$ and any $x_1,\ldots,x_s\in K$,
there exist $\epsilon_1,\ldots,\epsilon_s \in \{-1,1\}$ such that $\max_{k\ls s}\|\sum_{j=1}^k \epsilon_j x_j\|_K\ls 2n$.

The question that we discuss is whether one can achieve something better than the $O(n)$ bound for a random $s$-tuple
$(x_1,\ldots, x_s)$ from $C$. In order to make this question precise, for any $\delta\in (0,1)$ we introduce the parameter
\begin{equation*}\beta_{\delta ,s}^{(R)}(C,K):=\min\Big\{ r>0:{\rm vol}_{ns}\Big(\Big\{(x_j)_{j=1}^s:x_j\in C
\;\hbox{for all}\;j\;\hbox{and}\;\min_{\epsilon\in E_2^s}
\Big\|\sum_{j=1}^s\epsilon_jx_j\Big\|_K\ls r\Big\}\Big)\gr 1-\delta \Big\}.\end{equation*}
The results of Section~4 and Section~5 allow us to obtain significantly better bounds for $\beta_{\delta ,s}^{(R)}(C,K)$.
In the statement below we restrict ourselves to the case $C=K$ and $s=n$; the reader may deduce analogous
bounds for an arbitrary choice of $C$ or $s$.

\begin{theorem}\label{th:r-barany-grinberg}
Let $K$ be a centrally symmetric convex body in ${\mathbb R}^n$. Then, for any $\delta\in (0,1)$,
\begin{equation*}\beta_{\delta ,n}^{(R)}(K,K)\ls (c\log (2/\delta )L_Kn^{3/4})\,\sqrt{n}M(K_{{\rm iso}})\end{equation*}
where $c>0$ is an absolute constant and $K_{{\rm iso}}$ is an isotropic image of $K$. If $K$ is a $\psi_2$-body with constant $\varrho $ then
\begin{equation*}\beta_{\delta ,n}^{(R)}(K,K)\ls (c\log (2/\delta )\varrho^2 \sqrt{n})\,\sqrt{n}M(K_{{\rm iso}}).\end{equation*}
\end{theorem}

Analogous results hold for $2$-convex bodies with constant $\alpha $, in which case we have
\begin{equation*}\beta_{\delta ,n}^{(R)}(K,K)\ls (c\log (2/\delta )\sqrt{n}/\alpha ),\end{equation*}
or bodies with bounded cotype-$2$ constant; in this case we have
\begin{equation*}\beta_{\delta ,n}^{(R)}(K,K)\ls (c\log (2/\delta )L_KC_2(X_K)\sqrt{n})\,\sqrt{n}M(K_{{\rm iso}}).\end{equation*}
In fact, the proof of Theorem~\ref{th:r-barany-grinberg}
shows that the same upper bounds hold for the parameter $\kappa_{\delta ,s}^{(R)}(C,K)$ which is defined
as the smallest $r>0$ with the property that
\begin{equation*}{\rm vol}_{ns}\Big(\Big\{(x_j)_{j=1}^s:x_j\in C\;\hbox{for all}\;j\;\hbox{and}\; {\mathbb P}\Big(\Big\{ \epsilon\in E_2^s:
\Big\|\sum_{j=1}^s\epsilon_jx_j\Big\|_K\ls r\Big\}\Big)\gr 1-\delta \Big\}\Big)\gr 1-\delta .\end{equation*}
Note that, by definition, $\kappa_{\delta ,s}^{(R)}(C,K)\gr\beta_{\delta ,s}^{(R)}(C,K)$.

Finally, combining our approach with some classical results from asymptotic convex geometry we obtain variants of the
main results of \cite{Chasapis-Skarmogiannis} as well as their dual estimates. We close this introductory section with
the statements in the particular case $C=\overline{B_2^n}$.

\begin{theorem}\label{th:section-7}
Let ${\bf t}\in {\mathbb R}^s$. For any centrally symmetric convex body $K$ in ${\mathbb R}^n$ and
any $S\subseteq E_2^s$ with $|S|\ls e^{cd(K)}$ we have
\begin{equation*}\vol_{ns}\Big(\Big\{(x_j)_{j=1}^s:x_j\in \overline{B_2^n}\;\;\hbox{for all}\;j\;\hbox{and}\,
\,\Big\|\sum_{j=1}^s\epsilon_jt_jx_j\Big\|_K\ls c_1L_C\sqrt{n}M(K)\,\|{\bf t}\|_2\;\;\hbox{for some}\;\epsilon\in S\Big\}\Big )
\ls e^{-c_2d(K)}\end{equation*}
and
\begin{equation*}\vol_{ns}\Big(\Big\{(x_j)_{j=1}^s:x_j\in \overline{B_2^n}\;\;\hbox{for all}\;j\;\hbox{and}\;
\Big\|\sum_{j=1}^s\epsilon_jt_jx_j\Big\|_K\gr c_3L_C\sqrt{n}M(K)\,\|{\bf t}\|_2\;\;\hbox{for some}\;\epsilon\in S\Big\}\Big)
\ls e^{-c_4k(K)},\end{equation*}
where $c_i>0$ are absolute constants.
\end{theorem}

The quantities $k(K)$ and $d(K)$ are well-known parameters of a centrally symmetric convex body $K$ which are introduced in Section~7;
$k(K)=n(M(K)/b(K))^2$ is the Dvoretzky dimension of $K$ and
\begin{equation*}d(K)=\min\Big\{ n,-\log\gamma_n\Big(\frac{m(K)}{2}K\Big)\Big\},\end{equation*}
where $m(K)\approx \sqrt{n}M(K)$ is the median (the L\'{e}vy mean) of $\|\cdot\|_K$ with
respect to the standard Gaussian measure $\gamma_n$ on $\mathbb{R}^n$.

\section{Backgound information and preliminary observations}\label{sec:background}

In this section we introduce notation and terminology that we use throughout this work, and provide background
information on isotropic convex bodies. We write $\langle\cdot ,\cdot\rangle $ for the standard inner product in ${\mathbb R}^n$ and denote the
Euclidean norm by $\|\cdot \|_2$. In what follows, $B_2^n$ is the Euclidean unit ball, $S^{n-1}$ is the unit sphere, and $\sigma $ is the
rotationally invariant probability measure on $S^{n-1}$. Lebesgue measure in ${\mathbb R}^n$ is denoted by ${\rm vol}_n$.
The letters $c, c^{\prime },c_j,c_j^{\prime }$ etc. denote absolute positive constants whose value may change from line to line.
Sometimes we might even relax our notation: $a\lesssim b$ will then mean ``$a\ls cb$ for some (suitable) absolute constant $c>0$'',
and $a \approx b$ will stand for ``$a \lesssim b \land a \gtrsim b$". If $A, B$ are sets, $A \approx B$
will similarly state that $c_1A\subseteq B \subseteq c_2 A$ for some absolute constants $c_1,c_2>0$.

A convex body in ${\mathbb R}^n$ is a compact convex set $C\subset {\mathbb R}^n$ with non-empty interior. We say that $C$ is {\bf centrally symmetric}
if $-C=C$. We say that $C$ is unconditional with respect to the standard orthonormal basis $\{e_1,\ldots ,e_n\}$ of ${\mathbb R}^n$ if $x=(x_1,\ldots ,x_n)\in C$
implies that $(\epsilon_1x_1,\ldots ,\epsilon_nx_n)\in C$ for any choice
of signs $\epsilon_j\in\{ -1,1\}$. The volume radius of $C$ is the quantity ${\rm vrad}(C)=\left (\vol_n(C)/\vol_n(B_2^n)\right )^{1/n}$.
Integration in polar coordinates shows that if the origin is an interior point of $C$ then the volume radius of $C$ can be expressed as
\begin{equation*}{\rm vrad}(C)=\Big (\int_{S^{n-1}}\|\xi\|_C^{-n}\,d\sigma (\xi )\Big)^{1/n},\end{equation*}
where $\|\xi\|_C=\inf\{ t>0:\xi\in tC\}$. We also consider the parameter
\begin{equation*}M(C)=\int_{S^{n-1}}\|\xi\|_Cd\sigma (\xi ).\end{equation*}
The support function of $C$ is defined by $h_C(y):=\max \bigl\{\langle x,y\rangle :x\in C\bigr\}$, and
the mean width of $C$ is the average
\begin{equation*}w(C):=\int_{S^{n-1}}h_C(\xi)\,d\sigma (\xi)\end{equation*}
of $h_C$ on $S^{n-1}$. The radius $R(C)$ of a centrally symmetric convex body $C$ is the smallest $R>0$ such that $C\subseteq RB_2^n$.
We shall use the fact that $R(C)\ls c\sqrt{n}w(C)$; equivalently, $b(C)\ls c\sqrt{n}M(C)$, where
$b(C)$ is the smallest $b>0$ with the property that $\|x\|_C\ls b\|x\|_2$ for all $x\in {\mathbb R}^n$.
For notational convenience we write $\overline{C}$ for
the homothetic image of volume $1$ of a convex body $C\subseteq
\mathbb R^n$, i.e. $\overline{C}:= \vol_n(C)^{-1/n}C$.

The polar body $C^{\circ }$ of a centrally symmetric convex body $C$ in ${\mathbb R}^n$ is defined by
\begin{equation*}
C^{\circ}:=\bigl\{y\in {\mathbb R}^n: \langle x,y\rangle \ls 1\;\hbox{for all}\; x\in C\bigr\}.
\end{equation*}
The Blaschke-Santal\'{o} inequality states that ${\rm vol}_n(C){\rm vol}_n(C^{\circ })\ls {\rm vol}_n(B_2^n)^2$,
with equality if and only if $C$ is an ellipsoid. The reverse Santal\'{o} inequality of Bourgain and V. Milman
\cite{Bourgain-VMilman-1987} asserts that there exists an absolute constant $c>0$ such
that, conversely,
\begin{equation*}\Big ({\rm vol}_n(C){\rm vol}_n(C^{\circ })\Big )^{1/n}\gr c\,\vol_n(B_2^n)^{2/n}\approx 1/n.\end{equation*}
A convex body $C$ in ${\mathbb R}^n$ is called isotropic if it has volume $1$, it is centered, i.e.~its
barycenter is at the origin, and its inertia matrix is a multiple of the identity matrix:
there exists a constant $L_C >0$ such that
\begin{equation}\label{isotropic-condition}\|\langle \cdot ,\xi\rangle\|_{L_2(C)}^2:=\int_C\langle x,\xi\rangle^2dx =L_C^2\end{equation}
for all $\xi\in S^{n-1}$. We shall use the fact that if $C$ is isotropic then $R(C)\ls cnL_C$ for some
absolute constant $c>0$. The hyperplane conjecture asks if there exists an absolute constant $A>0$ such that
\begin{equation}\label{HypCon}L_n:= \max\{ L_C:C\ \hbox{is isotropic in}\ {\mathbb R}^n\}\ls A\end{equation}
for all $n\gr 1$. Bourgain proved in \cite{Bourgain-1991} that $L_n\ls c\sqrt[4]{n}\log\! n$; later, Klartag \cite{Klartag-2006}
improved this bound to $L_n\ls c\sqrt[4]{n}$.

A Borel measure $\mu$ on $\mathbb R^n$ is called $\log$-concave if $\mu(\lambda
A+(1-\lambda)B) \gr \mu(A)^{\lambda}\mu(B)^{1-\lambda}$ for any compact subsets $A$
and $B$ of ${\mathbb R}^n$ and any $\lambda \in (0,1)$. A function
$f:\mathbb R^n \rightarrow [0,\infty)$ is called $\log$-concave if
its support $\{f>0\}$ is a convex set and the restriction of $\log{f}$ to it is concave.
It is known that if a probability measure $\mu $ is log-concave and $\mu (H)<1$ for every
hyperplane $H$, then $\mu $ has a log-concave density $f_{{\mu }}$. Note that if $C$ is a convex body in
$\mathbb R^n$ then the Brunn-Minkowski inequality implies that
$\mathbf{1}_{C} $ is the density of a $\log$-concave measure.

If $\mu $ is a log-concave measure on ${\mathbb R}^n$ with density $f_{\mu}$, we define the isotropic constant of $\mu $ by
\begin{equation}\label{definition-isotropic}
L_{\mu }:=\left (\frac{\sup_{x\in {\mathbb R}^n} f_{\mu} (x)}{\int_{{\mathbb
R}^n}f_{\mu}(x)dx}\right )^{\frac{1}{n}} [\det \textrm{Cov}(\mu)]^{\frac{1}{2n}},\end{equation}
where $\textrm{Cov}(\mu)$ is the covariance matrix of $\mu$ with entries
\begin{equation}\textrm{Cov}(\mu )_{ij}:=\frac{\int_{{\mathbb R}^n}x_ix_j f_{\mu}
(x)\,dx}{\int_{{\mathbb R}^n} f_{\mu} (x)\,dx}-\frac{\int_{{\mathbb
R}^n}x_i f_{\mu} (x)\,dx}{\int_{{\mathbb R}^n} f_{\mu}
(x)\,dx}\frac{\int_{{\mathbb R}^n}x_j f_{\mu}
(x)\,dx}{\int_{{\mathbb R}^n} f_{\mu} (x)\,dx}.\end{equation} We say
that a $\log $-concave probability measure $\mu $ on ${\mathbb R}^n$
is isotropic if it is centered, i.e. if
\begin{equation}
\int_{\mathbb R^n} \langle x, \xi \rangle d\mu(x) = \int_{\mathbb
R^n} \langle x, \xi \rangle f_{\mu}(x) dx = 0
\end{equation} for all $\xi\in S^{n-1}$, and $\textrm{Cov}(\mu )$ is the identity matrix.

If $C$ is a centered convex body of volume $1$ in $\mathbb R^n$ then we say that a direction $\xi\in S^{n-1}$ is a
$\psi_{\alpha }$-direction (where $1\ls\alpha\ls 2$) for $C$ with constant $\varrho >0$ if
\begin{equation*}\|\langle\cdot ,\xi\rangle\|_{L_{\psi_{\alpha }}(C)}\ls \varrho\|\langle\cdot, \xi\rangle\|_{L_2(C)},\end{equation*}
where \begin{equation*} \|\langle \cdot ,\xi\rangle \|_{L_{\psi_{\alpha}}(C)}:=\inf \Big \{ t>0 :
\int_C\exp \big((|\langle x,\xi\rangle |/t)^\alpha\big) \,
dx\ls 2 \Big\}.
\end{equation*}From Markov's inequality it is clear that if $C$ satisfies a
$\psi_\alpha$-estimate with constant $\varrho $ in the direction of
$\xi $ then for all $t\gr 1$ we have $\vol_n(\{x\in C : |\langle x,\xi\rangle|\gr t \|\langle\cdot ,\xi\rangle\|_{L_2(C)}\})\ls
2e^{-t^a/\varrho^\alpha}$. Conversely, it is a standard fact that tail estimates of this form
imply that $\xi $ is a $\psi_{\alpha }$-direction for $C$.
Similar definitions may be given in the context of a centered log-concave probability measure
$\mu $ on ${\mathbb R}^n$. From log-concavity it follows that every $\xi \in S^{n-1}$ is a $\psi_1$-direction
for any $C$ or $\mu $ with an absolute constant $\varrho $: there exists $\varrho >0$ such that
\begin{equation*}\|\langle\cdot ,\xi\rangle\|_{L_{\psi_1}(\mu )}\ls \varrho\|\langle\cdot, \xi\rangle\|_{L_2(\mu )}\end{equation*}
for all $n\gr 1$, all centered log-concave probability measures $\mu $ on ${\mathbb R}^n$ and all $\xi\in S^{n-1}$.
We refer the reader to the book \cite{BGVV} for an updated exposition of isotropic log-concave
measures and more information on the hyperplane conjecture.

We close this introductory section with a lemma that may be viewed as a form of generalization of Khinchine's inequality, where the randomness is no longer
that of Bernoulli $\{-1,1\}$ random variables but here is given by random vectors in the bodies $C_1,\ldots ,C_s$.

\begin{lemma}\label{lem:higher-moments}
Let $C_1,\ldots ,C_s$ be convex bodies of volume $1$ and $K$ be a centrally symmetric convex body in $\mathbb{R}^n$. Then,
\begin{equation*}\Big ({\mathbb E}_{{\cal C}}\Big\|\sum_{j=1}^st_jx_j\Big\|_K^q\Big )^{1/q}\ls cq\,\|{\bf t}\|_{{\cal C},K},\end{equation*}
where $c>0$ is an absolute constant.
\end{lemma}

The lemma follows immediately from the fact (see \cite[Theorem~2.4.6]{BGVV}) that if $\mu $ is a log-concave
probability measure on ${\mathbb R}^k$ and $f:{\mathbb R}^k\to {\mathbb R}$ is a seminorm then, for any $q\gr 1$,
\begin{equation*}\|f\|_{L_q(\mu )}\ls cq\,\|f\|_{L_1(\mu )},\end{equation*}
where $c>0$ is an absolute constant. We apply this fact on ${\mathbb R}^{ns}$ for the semi-norm
\begin{equation*}(x_1,\ldots ,x_s)\mapsto \Big\|\sum_{j=1}^st_jx_j\Big\|_K\end{equation*}
and the uniform measure on $C_1\times\cdots\times C_s$.

\section{A basic identity and a proof of the lower bound}

In this section we assume that $C_1,\ldots ,C_s$ are centrally symmetric convex bodies of volume $1$ in ${\mathbb R}^n$ and study the quantity
\begin{equation*}\|{\bf t}\|_{{\cal C},K}=\int_{C_1}\cdots\int_{C_s}\Big\|\sum_{j=1}^st_jx_j\Big\|_K\,dx_1\cdots dx_s\end{equation*}
where ${\bf t}=(t_1,\ldots ,t_s)$ and $K$ is a centrally symmetric convex body in ${\mathbb R}^n$.  By the symmetry of the $C_j$'s we have that
\begin{equation*}\int_{C_1}\cdots\int_{C_s}\Big\|\sum_{j=1}^st_jx_j\Big\|_K\,dx_1\cdots dx_s
=\int_{C_1}\cdots\int_{C_s}\Big\|\sum_{j=1}^s\epsilon_jt_jx_j\Big\|_K\,dx_1\cdots dx_s\end{equation*}
for all $\epsilon =(\epsilon_1,\ldots ,\epsilon_s)\in E_2^s$, therefore we may always assume that
$t_1,\ldots ,t_s\gr 0$. Our starting point is the next observation.

\begin{lemma}\label{lem:identity}
Let $X_1,\ldots ,X_s$ be independent random vectors, uniformly distributed on $C_1,\ldots ,C_s$ respectively. Given ${\bf t}=(t_1\ldots ,t_s)\in {\mathbb R}^s$,
we write $\nu_{{\bf t}}$ for the distribution of the random vector $t_1X_1+\cdots +t_sX_s$. Then,
\begin{equation*}\|{\bf t}\|_{{\cal C},K}=\int_{{\mathbb R}^n}\|x\|_Kd\nu_{{\bf t}}(x).\end{equation*}
\end{lemma}

Since $\|{\bf t}\|_{{\cal C},K}$ is a norm, we may always assume that $\|{\bf t}\|_2=1$.
Note that $\nu_{{\bf t}}$ is an even log-concave probability measure on ${\mathbb R}^n$
(this is a consequence of the Pr\'{e}kopa-Leindler inequality; see \cite{AGA-book}). We write $g_{{\bf t}}$ for the density of $\nu_{{\bf t}}$.
The next lemma provides an upper bound for $\|g_{{\bf t}}\|_{\infty }=g_{{\bf t}}(0)$.

\begin{lemma}\label{lem:bound-gt}If $\|{\bf t}\|_2=1$ then $\|g_{{\bf t}}\|_{\infty }\ls e^n$.\end{lemma}

\begin{proof}The proof employs a result of Bobkov and Madiman from \cite{Bobkov-Madiman-2011} and the Shannon-Stam inequality (see \cite{Stam-1959}).
Recall that the entropy functional of a random vector $X$ in ${\mathbb R}^n$ with density $g(x)$ is defined by
\begin{equation*}h(X)=-\int_{{\mathbb R}^n}g(x)\log g(x)\,dx\end{equation*}
provided the integral exists. Bobkov and Madiman have shown that if $g$ is log-concave then
\begin{equation*}\log (\|g\|_{\infty }^{-1})\ls h(X)\ls n+ \log (\|g\|_{\infty }^{-1})\end{equation*}
(the assumption that $g$ is log-concave is needed only for the right hand side inequality). Let ${\bf t}\in {\mathbb R}^s$ with
$\|{\bf t}\|_2=1$ and $t_1,\ldots ,t_s\gr 0$. Then, if $X_1,\ldots ,X_s$ are independent random vectors with densities $g_1,\ldots ,g_s$ we have that
\begin{equation*}h(t_1X_1+\cdots +t_sX_s)\gr \sum_{j=1}^st_j^2h(X_j).\end{equation*}
This is an equivalent form of the Shannon-Stam inequality (see \cite{Lieb-1978} and \cite{Dembo-Cover-Thomas-1991}).
Since the density $g_{{\bf t}}$ of $t_1X_1+\cdots +t_sX_s$ is also log-concave, we may write
\begin{equation*}\sum_{j=1}^st_j^2\log (\|g_j\|_{\infty }^{-1})\ls\sum_{j=1}^st_j^2h(X_j)\ls h(t_1X_1+\cdots +t_sX_s)\ls n+\log (\|g_{{\bf t}}\|_{\infty }^{-1}),\end{equation*}
which implies that
\begin{equation*}\|g_{{\bf t}}\|_{\infty }\ls e^n\prod_{j=1}^s\|g_j\|_{\infty }^{t_j^2}.\end{equation*}
In our case, $g_j={\bf 1}_{C_j}$, therefore $\|g_j\|_{\infty }=1$ and the lemma follows.
\end{proof}

The next lemma is an immediate consequence of \cite[Lemma~2.3]{BMMP} (see also \cite[Lemma~2.1]{VMilman-Pajor-1989}).

\begin{lemma}\label{lem:BMMP}Let $f$ be a bounded positive density of a probability measure $\mu $ on ${\mathbb R}^n$. For any centrally symmetric
convex body $K$ in ${\mathbb R}^n$ and any $p>0$ one has
\begin{equation*}\left (\frac{n}{n+p}\right )^{1/p}\ls\left (\int_{{\mathbb R}^n}\|x\|_K^pf(x)\,dx\right )^{1/p}\|f\|_{\infty }^{1/n}{\rm vol}_n(K)^{1/n}.\end{equation*}
\end{lemma}

We apply Lemma~\ref{lem:BMMP} for the log-concave probability measure $\nu_{{\bf t}}$. For any ${\bf t}\in {\mathbb R}^s$
with $\|{\bf t}\|_2=1$ we have $\|g_{{\bf t}}\|_{\infty }=g_{{\bf t}}(0)\ls e^n$, therefore
\begin{equation*}\frac{n}{n+1}\ls e\,{\rm vol}_n(K)^{1/n}\int_{{\mathbb R}^n}\|x\|_K\,d\nu_{{\bf t}}(x).\end{equation*}

Combining this inequality with Lemma \ref{lem:identity} we see that
if ${\cal C}=(C_1,\ldots ,C_s)$ is an $s$-tuple of centrally symmetric convex bodies of volume $1$ and $K$ is a centrally symmetric convex body in $\mathbb{R}^n$
then, for any $s\gr 1$ and any ${\bf t}=(t_1,\ldots,t_s)\in \mathbb{R}^s$
\begin{equation*}\|{\bf t}\|_{{\cal C},K}\gr \frac{n}{e(n+1)}{\rm vol}_n(K)^{-1/n}\, \|{\bf t}\|_2.\end{equation*}
This proves Theorem~\ref{th:GM-short}.

\section{Upper bounds}

In this section we assume that $C$ is an isotropic convex body in ${\mathbb R}^n$. We shall further exploit the identity of Lemma \ref{lem:identity} to
give upper estimates for $\|{\bf t}\|_{C^s,K}$, where $K$ is a centrally symmetric convex body in ${\mathbb R}^n$.

As in the previous section, let $X_1,\ldots ,X_s$ be independent random vectors, uniformly distributed on $C$.
Given ${\bf t}=(t_1\ldots ,t_s)\in {\mathbb R}^s$ with $\|{\bf t}\|_2=1$,
we write $\nu_{{\bf t}}$ for the distribution of the random vector $t_1X_1+\cdots +t_sX_s$. It is then easily verified that the
covariance matrix ${\rm Cov}(\nu_{{\bf t}})$ of $\nu_{{\bf t}}$ is a multiple of the identity: more precisely,
\begin{equation*}{\rm Cov}(\nu_{{\bf t}})=L_C^2\,I_n.\end{equation*}
It follows that if $g_{{\bf t}}$ is the density of $\nu_{{\bf t}} $ then $f_{{\bf t}}(x)=L_C^ng_{{\bf t}}(L_Cx)$ is the density
of an isotropic log-concave probability measure on ${\mathbb R}^n$. Indeed, we have
\begin{equation*}\int_{{\mathbb R}^n}f_{{\bf t}}(x)x_ix_j\,dx=L_C^n\int_{{\mathbb R}^n}g_{{\bf t}}(L_Cx)x_ix_j\,dx
=L_C^{-2}\int_{{\mathbb R}^n}g_{{\bf t}}(y)y_iy_j\,dy=\delta_{ij}\end{equation*}
for all $1\ls i,j\ls n$. From Lemma~\ref{lem:bound-gt} we see that
\begin{equation*}L_{\mu_{{\bf t}}}=\|f_{{\bf t}}\|_{\infty }^{\frac{1}{n}}=L_C\|g_{{\bf t}}\|_{\infty }^{\frac{1}{n}}\ls eL_C\end{equation*}
for all ${\bf t}\in {\mathbb R}^s$ with $\|{\bf t}\|_2=1$. We also have
\begin{equation*}\|{\bf t}\|_{C^s,K}=\int_{{\mathbb R}^n}\|x\|_K\,d\nu_{{\bf t}}(x)=L_C^{-n}\int_{{\mathbb R}^n}\|x\|_Kf_{{\bf t}}(x/L_C)\,dx
= L_C\int_{{\mathbb R}^n}\|y\|_Kd\mu_{{\bf t}}(y).\end{equation*}

\begin{definition}\label{def:I-one}\rm Let $\mu $ be a centered log-concave probability measure on ${\mathbb R}^n$. For any star body $K$ in ${\mathbb R}^n$
we define
\begin{equation*}I_1(\mu ,K)=\int_{{\mathbb R}^n}\|x\|_Kd\mu (x).\end{equation*}
With this definition, we can write
\begin{equation}\label{eq:isotropic-identity}\|{\bf t}\|_{C^s,K}=L_C\, I_1(\mu_{{\bf t}},K)\end{equation}
for all ${\bf t}\in {\mathbb R}^s$ with $\|{\bf t}\|_2=1$. Then, our aim is to establish an upper bound for $I_1(\mu_{{\bf t}},K)$.
\end{definition}

\subsection{Simple upper and lower bounds}

A first upper bound for $I_1(\mu_{{\bf t}},K)$ can be obtained if we use the simple inequality $\|y\|_K\ls b\|y\|_2$, where $b=b(K)=R(K^{\circ })$.
We observe that
\begin{equation*}
I_1(\mu_{{\bf t}},K) \ls b\int_{{\mathbb R}^n}\|y\|_2d\mu_{{\bf t}}(y)\ls b\sqrt{n}.
\end{equation*}
because the last integral is bounded by $\sqrt{n}$: this follows immediately from the Cauchy-Schwarz inequality and the
isotropicity of $\mu_{{\bf t}}$. On the other hand,
\begin{align*}
I_1(\mu_{{\bf t}},K) &=\int_{{\mathbb R}^n}\max_{x\in K^{\circ }}|\langle x,y\rangle |\,d\mu_{{\bf t}}(y)
\gr \max_{x\in K^{\circ }}\int_{{\mathbb R}^n}|\langle x,y\rangle |\,d\mu_{{\bf t}}(y)
\gr \max_{x\in K^{\circ }}c_1\Big (\int_{{\mathbb R}^n}|\langle x,y\rangle |^2\,d\mu_{{\bf t}}(y)\Big )^{1/2}\\
&= c_1\, \max_{x\in K^{\circ }}\|x\|_2=c_1R(K^{\circ })=c_1b,
\end{align*}
where in the second inequality we are using \cite[Theorem~2.4.6]{BGVV}. Inserting these two bounds
into \eqref{eq:isotropic-identity} we have the next theorem.

\begin{theorem}\label{theorem:simple-ell2-bound}
Let $C$ be an isotropic convex body in $\mathbb{R}^n$ and $K$ be a centrally symmetric convex body in $\mathbb{R}^n$. Then, for any $s\gr 1$
and ${\bf t}=(t_1,\ldots,t_s)\in \mathbb{R}^s$,
\begin{equation*}
c_1L_CR(K^{\circ })\,\|{\bf t}\|_2\ls \|{\bf t}\|_{C^s,K} \ls \sqrt{n}L_CR(K^{\circ })\,\|{\bf t}\|_2,
\end{equation*}
where $c_1>0$ is an absolute constant.
\end{theorem}

There are some classes of centrally symmetric convex bodies that behave well with respect to the upper bound
of Theorem~\ref{theorem:simple-ell2-bound}. We discuss one of them in the next subsection.

\subsection{$2$-convex bodies}

Recall that if $K$ is a centrally symmetric convex body in $\mathbb{R}^n$ then the modulus of convexity of $K$ is the function
$\delta_K:(0,2]\to {\mathbb R}$ defined by
\begin{equation*}
\delta_K(\varepsilon)=\inf\Big\{1-\Big\|\frac{x+y}{2}\Big\|_K\,:\, \|x\|_K, \|y\|_K\ls 1, \|x-y\|_K\gr\varepsilon\Big\}.
\end{equation*}
Then, $K$ is called $2$-convex with constant $\alpha$ if, for every $\varepsilon\in (0,2]$,
\begin{equation*}
\delta_K(\varepsilon)\gr \alpha\varepsilon^2.
\end{equation*}
Examples of $2$-convex bodies are given by the unit balls of subspaces of $L_p$-spaces, $1<p\ls 2$; one can check that
the definition is satisfied with $\alpha\approx p-1$. Klartag and E.~Milman have proved in
\cite{Klartag-EMilman-2008} that if $K$ is a centrally symmetric convex body of volume $1$ in ${\mathbb R}^n$,
which is also $2$-convex with constant $\alpha$, then
\begin{equation*}
L_K\ls c_1/\sqrt{\alpha},
\end{equation*}
where $c_1>0$ is an absolute constant. Moreover, if $K$ is isotropic then
\begin{equation*}
c_2\sqrt{\alpha}\sqrt{n}B_2^n\subseteq K,
\end{equation*}
for an absolute constant $c_2>0$ (see, again, \cite{Klartag-EMilman-2008}). From Theorem~\ref{theorem:simple-ell2-bound} we immediately get the next estimate.

\begin{theorem}\label{thm.2.convex}
Let $C$ be an isotropic convex body in $\mathbb{R}^n$ and $K$ be an isotropic centrally symmetric convex body in ${\mathbb R}^n$ which is also
$2$-convex with constant $\alpha$. Then for any $s\gr 1$ and ${\bf t}=(t_1,\ldots,t_s)\in \mathbb{R}^s$,
\begin{equation*}
\|{\bf t}\|_{C^s,K} \ls \frac{cL_C}{\sqrt{\alpha }}\|{\bf t}\|_2
\end{equation*}
where $c>0$ is an absolute constant. In particular, for any centrally symmetric convex body $K$ in ${\mathbb R}^n$ which is $2$-convex
with constant $\alpha $, we have that
\begin{equation*}
\|{\bf t}\|_{K^s,K} \ls \frac{c}{\alpha }\|{\bf t}\|_2.
\end{equation*}
\end{theorem}

\begin{proof}The first claim follows from the fact that $R(K^{\circ })\ls c_2^{-1}/(\sqrt{\alpha}\sqrt{n})$. For the second assertion
we use the observation that $\mathbb{E}_{K^s}\left\|\sum_{j=1}^s t_j x_j \right\|_K=\mathbb{E}_{(TK)^s}\left\|\sum_{j=1}^s t_j x_j \right\|_{TK}$
for any $T\in SL(n)$, and hence we may assume that $K$ is isotropic. Since $L_K\ls c_1/\sqrt{\alpha }$ we see that
\begin{equation*}
\mathbb{E}_{K^s}\Big\|\sum_{j=1}^s t_j x_j \Big\|_K \ls \frac{c_2^{-1}L_K}{\sqrt{\alpha }}\|{\bf t}\|_2\ls \frac{c_3}{\alpha }\|{\bf t}\|_2,
\end{equation*}
where $c_3=c_2^{-1}c_1$.\end{proof}

\subsection{A general upper bound}

In this subsection we prove Theorem~\ref{th:general}. By homogeneity it is enough to consider the case $\|{\bf t}\|_2=1$.
Our starting point will be again \eqref{eq:isotropic-identity}. We have
\begin{equation*}\|{\bf t}\|_{C^s,K}=L_C\, I_1(\mu_{{\bf t}},K),\end{equation*}
and hence our aim is to establish an upper bound for $I_1(\mu_{{\bf t}},K)$. We shall use a well-known inequality of Paouris from
\cite{Paouris-1}.

\begin{theorem}[Paouris]\label{th:paouris}
If $\mu $ is an isotropic log-concave probability measure on $\mathbb{R}^n$, then
\begin{equation*}\mu (\{x\in {\mathbb R}^n : \|x\|_2\gr c_1\, r\sqrt{n}\}) \ls e^{-r\sqrt{n}}\end{equation*}
for every $r\gr 1$, where $c_1>0$ is an absolute constant.
\end{theorem}

Note also that, since $R(C)\ls c_2nL_C$ and ${\rm supp}(\nu_{{\bf t}})\subseteq sC$, we have that
\begin{equation*}{\rm supp}(\mu_{{\bf t}})\subseteq \frac{s}{L_C}\,C\subseteq (c_2ns)\,B_2^n\end{equation*}
for any ${\bf t}=(t_1,\ldots,t_s)\in \mathbb{R}^s$ with $\|{\bf t}\|_2=1$. Therefore, if we fix $r\gr 1$ and set
$C_{{\bf t}}(r)={\rm supp}(\mu_{{\bf t}})\cap c_1r\sqrt{n}B_2^n$, we may write
\begin{align*}\int_{{\mathbb R}^n}\|x\|_K\,d\mu_{{\bf t}}(x) &= \int_{C_{{\bf t}}(r)}\|x\|_K\,d\mu_{{\bf t}}(x)+
\int_{{\rm supp}(\mu_{{\bf t}})\setminus C_{{\bf t}}(r)}\|x\|_K\,d\mu_{{\bf t}}(x)\\
&\ls \int_{C_{{\bf t}}(r)}\|x\|_K\,d\mu_{{\bf t}}(x)+b(K)\int_{{\rm supp}(\mu_{{\bf t}})\setminus C_{{\bf t}}(r)}\|x\|_2d\mu_{{\bf t}}(x)\\
&\ls \int_{C_{{\bf t}}(r)}\|x\|_K\,d\mu_{{\bf t}}(x)+b(K)\,(c_2ns) \,e^{-r\sqrt{n}}.
\end{align*}
Turning our attention to the first term, we consider the log-concave probability measure $\mu_{{\bf t},r}$ with density
\begin{equation*}\frac{1}{\mu_{{\bf t}}(C_{{\bf t}}(r))}\,{\mathbf 1}_{C_{{\bf t}}(r)}f_{{\bf t}}\end{equation*}
and the stochastic process $(w_y)_{y\in K^\circ}$ on $({\mathbb R}^n,\mu_{{\bf t},r})$, where $w_y(x)=\langle x,y \rangle$.
We also consider a standard Gaussian random vector $G$ in $\mathbb{R}^n$, and for $y\in K^\circ$ set $h_y(G) = \langle G,y \rangle$. Note that
(see e.g. \cite[Lemma~9.1.3]{AGA-book})
\begin{equation}\label{eq.bound.gaussian}
\mathbb{E}\,\Big(\max_{y\in K^\circ}h_y(G)\Big) = \mathbb{E}\,\|G\|_K \approx \sqrt{n}M(K).
\end{equation}
To bound $\mathbb{E}(\max_{y\in K^\circ}w_y)$, we will use Talagrand's comparison theorem (see \cite{Talagrand-1987}).

\begin{theorem}[Talagrand's comparison theorem]\label{thm.Talagrand}
If $(Y_t)_{t\in T}$ is a Gaussian process and $(X_t)_{t\in T}$ is a stochastic process such that
\begin{equation*}
\|X_s-X_t\|_{\psi_2} \ls \alpha \,\|Y_s-Y_t\|_2
\end{equation*}
for some $\alpha >0$ and every $s, t\in T$, then
\begin{equation*}
\mathbb{E}\,\Big (\max_{t\in T} X_t\Big) \ls c\alpha\, \mathbb{E} \Big(\max_{t\in T} Y_t\Big).
\end{equation*}
\end{theorem}

It is easily checked that $\|h_y-h_z\|_2=\|y-z\|_2$ for all $y,z\in K^{\circ }$. To bound the $\psi_2$ norm of $w_y-w_z$, we use the inequality $\|h\|_{\psi_2}\ls \sqrt{\|h\|_{\psi_1}\|h\|_\infty}$. Note that
\begin{equation*}\|w_y-w_z\|_{L^\infty(\mu_{{\bf t},r})} \ls R(C_{{\bf t}}(r))\|y-z\|_2\ls c_1r\sqrt{n}\|y-z\|_2\end{equation*}
and we also have
\begin{equation*}\|w_y-w_z\|_{L^{\psi_1}(\mu_{{\bf t},r})} \ls c_3\|w_y-w_z\|_{L^2(\mu_{{\bf t},r})} \ls 2c_3\|y-z\|_2\end{equation*}
for some absolute constant $c_3>0$ (here we also use the fact that $\mu (C_{{\bf t}}(r))\gr 1-e^{-r\sqrt{n}}\gr 1/2$). It follows that
\begin{equation*}
\|w_y-w_z\|_{L^{\psi_2}(\mu_{{\bf t},r})} \ls c_4\sqrt{r}\,\sqrt[4]{n}\,\|h_y-h_z\|_2.\end{equation*}
Theorem \ref{thm.Talagrand} then implies that
\begin{align*}
\int_{C_{{\bf t}}(r)}\|x\|_K\,d\mu_{{\bf t}}(x) &= \mu_{{\bf t}}(C_{{\bf t}}(r))\,{\mathbb E}_{\mu_{{\bf t},r}}\,\Big(\max_{y\in K^\circ}w_y\Big)
\ls c_5\sqrt{r}\,\sqrt[4]{n}\,{\mathbb E}\,\Big(\max_{y\in K^\circ}h_y\Big) \\
&\approx \sqrt{r}\sqrt[4]{n}\,\sqrt{n}M(K).
\end{align*}
Finally,
\begin{equation*}\int_{{\mathbb R}^n}\|x\|_K\,d\mu_{{\bf t}}(x) \ls c_1^{\prime }\Big (\sqrt{r}\sqrt[4]{n}\,\sqrt{n}M(K)
+b(K)\,ns\,e^{-r\sqrt{n}}\Big ).\end{equation*}
Since $b(K)\ls c_6\sqrt{n}M(K)$ we have that
\begin{equation*}b(K)\,ns \,e^{-r\sqrt{n}}\ls c_6nse^{-r\sqrt{n}}\,\sqrt{n}M(K)\ls \sqrt{r}\,\sqrt[4]{n}\,\sqrt{n}M(K)\end{equation*}
if we choose
\begin{equation*}r\approx\max\Big\{ 1,\frac{\log (1+s)}{\sqrt{n}}\Big\}.\end{equation*}
Therefore,
\begin{equation*}\|{\bf t}\|_{C^s,K}=L_C\, I_1(\mu_{{\bf t}},K)\ls
\Big (c_2^{\prime }L_C\max\Big\{ 1,\frac{\sqrt{\log (1+s)}}{\sqrt[4]{n}}\Big\}\,\sqrt[4]{n}\Big )\,\sqrt{n}M(K)\end{equation*}
as claimed. $\hfill\Box $

\medskip

Adapting the proof of Theorem~\ref{th:general} we can show that if $C$ is assumed a $\psi_2$-body with constant $\varrho $, which means that every direction
$\xi $ is a $\psi_2$-direction for $C$ with constant $\varrho $, then a much better estimate is available.

\begin{theorem}\label{th:psi2-case}
Let $C$ be an isotropic convex body in $\mathbb{R}^n$, which is a $\psi_2$-body with constant $\varrho $, and $K$ be a
centrally symmetric convex body in $\mathbb{R}^n$. Then for any $s\gr 1$ and every ${\bf t}=(t_1,\ldots,t_s)\in \mathbb{R}^s$,
\begin{equation*}
\|{\bf t}\|_{C^s,K} \ls c\varrho^2\sqrt{n}M(K)\,\|{\bf t}\|_2
\end{equation*}
where $c>0$ is an absolute constant.
\end{theorem}

\begin{proof}We consider the Gaussian process $h_y(G) = \langle G,y \rangle $, where $G$ is a standard Gaussian random vector
in $\mathbb{R}^n$, and recall that $\|h_y-h_z\|_2=\|y-z\|_2$ and
\begin{equation*}
\mathbb{E}\,\Big(\max_{y\in K^\circ}h_y(G)\Big) \approx \sqrt{n}M(K).
\end{equation*}
The main observation is that if $\|{\bf t}\|_2=1$ then $\mu_{{\bf t}}$ is a $\psi_2$-measure with constant $\varrho $.
Indeed, for any $\xi\in S^{n-1}$ we have (see \cite[Proposition~2.6.1]{Vershynin.HDP}) that if $w_{\xi }(x)=\langle x,\xi\rangle $ then
\begin{equation*}\|\langle x,\xi\rangle \|_{L_{\psi_2}(\mu_{{\bf t}})}^2=\Big\|\Big\langle\sum_{j=1}^s L_C^{-1}t_jX_j,\xi\Big\rangle \Big\|_{L_{\psi_2}(C^s)}^2
\ls\sum_{j=1}^sL_C^{-2}t_j^2\|\langle X_j,\xi\rangle\|_{L_{\psi_2}(C)}^2\ls \varrho^2,\end{equation*}
and hence, for any $y,z\in K^{\circ }$, the $\psi_2$ norm of $w_y-w_z$ can be directly estimated as follows:
\begin{equation*}
\|w_y-w_z\|_{\psi_2}\ls c_1\varrho\|y-z\|_2=c_1\varrho \|h_y-h_z\|_2.
\end{equation*}
Then, Theorem~\ref{thm.Talagrand} and the fact that $L_C\ls c_2\varrho $ (see \cite[Chapter~7]{BGVV}) imply that
\begin{equation*}
\|{\bf t}\|_{C^s,K}=L_C\int_{{\mathbb R}^n}\|x\|_K\,d\mu_{{\bf t}}(x)=L_C\,{\mathbb E}\,\Big(\max_{y\in K^{\circ }}w_y\Big)
\ls c_3\varrho^2\, \mathbb{E}\Big(\max_{y\in K^\circ}h_y(G)\Big) \approx \varrho^2\sqrt{n}M(K).
\end{equation*}
as claimed. \end{proof}

\section{Bodies with bounded cotype-$2$ constant}

Let $K$ be a centrally symmetric convex body in ${\mathbb R}^n$. Recall that if $X_K$ is the normed space with unit ball $K$,
we write $C_{2,k}(X_K)$ for the best constant $C>0$ such that
\begin{align*}
\Big ({\mathbb E}_{\epsilon }\Big \| \sum_{i=1}^k \epsilon_i x_i\Big\|_K^2\Big)^{1/2}
\gr \frac{1}{C} \Big(\sum_{i=1}^k \|x_i\|_K^2\Big)^{1/2}
\end{align*}for all $x_1,\ldots, x_k\in X$. Then, the cotype-2 constant of $X_K$
is defined as $C_2(X_K):=\sup_k C_{2,k}(X_K)$. Replacing the $\epsilon_j$'s by independent standard Gaussian random variables
$g_j$ in the definition above, one may define the Gaussian cotype-$2$ constant $\alpha_2(X_K)$ of $X_K$. One can check that
$\alpha_2(X_K)\ls C_2(X_K)$. E.~Milman has proved in \cite{EMilman-2006} that if $\mu$ is a finite, compactly supported isotropic measure
on $\mathbb R^n$ then, for any centrally symmetric convex body $K$ in $\mathbb R^n$,
\begin{equation}\label{eq:Emanuel-1}
I_1(\mu, K)\ls c_1\alpha_2(X_K)\sqrt{n}M(K)\ls c_1C_2(X_K)\sqrt{n}M(K).
\end{equation}
Using \eqref{eq:Emanuel-1} we can prove the following.

\begin{theorem}\label{th:third}
Let $C$ be an isotropic centrally symmetric convex body in $\mathbb{R}^n$ and $K$ be a centrally symmetric convex body in ${\mathbb R}^n$. Then for any $s\gr 1$
and ${\bf t}=(t_1,\ldots,t_s)\in \mathbb{R}^s$,
\begin{equation*}
\frac{c_1}{C_2(X_K)}{\rm vol}_n(K)^{-1/n}\|{\bf t}\|_2\ls \mathbb{E}_{C^s}\Big\|\sum_{j=1}^s t_j x_j \Big\|_K
\ls \big (c_2L_CC_2(X_K)\sqrt{n}M(K)\big)\,\|{\bf t}\|_2
\end{equation*}
 where $c_1, c_2>0$ are absolute constants. In particular, for any
centrally symmetric convex body $K$ of volume $1$ in ${\mathbb R}^n$ we have that
\begin{equation*}
\frac{c_1}{C_2(X_K)}\|{\bf t}\|_2\ls \mathbb{E}_{K^s}\Big\|\sum_{j=1}^s t_j x_j \Big\|_K \ls \big (c_2L_KC_2(X_K)\sqrt{n}M(K_{{\rm iso}})\big)\,\|{\bf t}\|_2,
\end{equation*}
where $K_{{\rm iso}}$ is an isotropic image of $K$.
\end{theorem}

\begin{proof}
Combining \eqref{eq:Emanuel-1} with \eqref{eq:isotropic-identity} we get
\begin{equation*}\|{\bf t}\|_{C^s,K}\ls c_1L_CC_2(X_K)\,\sqrt{n}M(K)\end{equation*}
for all ${\bf t}\in {\mathbb R}^s$ with $\|{\bf t}\|_2=1$. On the other hand, for any ${\bf t}\in {\mathbb R}^s$,
by the symmetry of $C$ we have that
\begin{align*}\|{\bf t}\|_{C^s,K} &= \int_C\cdots \int_C\int_{E_2^s}\Big\|\sum_{j=1}^s\epsilon_jt_jx_j\Big\|_Kd\mu_s(\epsilon )\,dx_1\cdots dx_s\\
&\gr \int_C\cdots \int_C\frac{1}{\sqrt{2}}\Big(\int_{E_2^s}\Big\|\sum_{j=1}^s\epsilon_jt_jx_j\Big\|_K^2d\mu_s(\epsilon )\Big)^{1/2}\,dx_1\cdots dx_s\\
&\gr \frac{1}{\sqrt{2}C_2(X_K)}\int_C\cdots \int_C\Big (\sum_{j=1}^st_j^2\|x_j\|_K^2\Big )^{1/2}\,dx_1\cdots dx_s\\
&\gr \frac{c}{C_2(X_K)}\Big (\sum_{j=1}^st_j^2\int_C\|x_j\|_K^2dx_j\Big )^{1/2}\\
&\gr \frac{c}{C_2(X_K)}\|{\bf t}\|_2\int_C\|x\|_K\,dx,
\end{align*}
where $c>0$ is an absolute constant (in the first inequality we are using the Kahane-Khintchine inequality
and in the third inequality we are using \cite[Theorem~2.4.6]{BGVV}
for the semi-norm $(x_1,\ldots ,x_s)\mapsto \left (\sum_{j=1}^st_j^2\|x_j\|_K^2\right )^{1/2}$ on $C^s$, while in
the last step we are using the Cauchy-Schwarz inequality for $\|\cdot\|_K$ on $C$). From Lemma~\ref{lem:BMMP} with $f={\bf 1}_C$ we see that
\begin{equation*}\int_C\|x\|_K\,dx\gr \frac{n}{n+1}{\rm vol}_n(K)^{-1/n},\end{equation*}
and the result follows.

In the case $C=K$, we may assume that $K$ is isotropic and these bounds take the form
\begin{equation*}
\frac{c_1}{C_2(X_K)}\|{\bf t}\|_2\ls \mathbb{E}_{K^s}\Big\|\sum_{j=1}^s t_j x_j \Big\|_K \ls \big (c_2L_KC_2(X_K)\sqrt{n}M(K)\big)\,\|{\bf t}\|_2.
\end{equation*}
This completes the proof.
\end{proof}

\begin{remark}\rm Another interesting case is when $K$ has bounded type-$2$ constant. Recall that if $X_K$ is the normed space with unit ball $K$,
we write $T_{2,k}(X_K)$ for the best constant $T>0$ such that
\begin{align*}
\Big ({\mathbb E}_{\epsilon }\Big \| \sum_{i=1}^k \epsilon_i x_i\Big\|_K^2\Big)^{1/2}
\ls T\Big(\sum_{i=1}^k \|x_i\|_K^2\Big)^{1/2}
\end{align*}for all $x_1,\ldots, x_k\in X$. Then, the type-2 constant of $X_K$
is defined as $T_2(X_K):=\sup_k T_{2,k}(X_K)$. E.\ Milman has proved in \cite{EMilman-2006} that if $\mu$ is a finite, compactly supported isotropic measure
on $\mathbb R^n$ then, for any centrally symmetric convex body $K$ in $\mathbb R^n$,
\begin{equation}
I_1(\mu, K)\gr c\sqrt{n}\frac{M(K)}{T_2(X_K)}.
\end{equation}
Using this inequality and following a similar argument, as in the cotype-$2$ case, we get:

\begin{theorem}\label{th:second}
Let $C$ be an isotropic centrally symmetric convex body in $\mathbb{R}^n$ and $K$ be a centrally symmetric convex body in ${\mathbb R}^n$. Then for any $s\gr 1$
and ${\bf t}=(t_1,\ldots,t_s)\in \mathbb{R}^s$,
\begin{equation*}
\frac{c_1L_C\sqrt{n}M(K)}{T_2(X_K)}\,\|{\bf t}\|_2\ls \mathbb{E}_{C^s}\Big\|\sum_{j=1}^s t_j x_j \Big\|_K \ls c_2T_2(X_K)\Big (\int_C\|x\|_Kdx\Big)
\,\|{\bf t}\|_2\end{equation*}
where $c_1,c_2>0$ are absolute constants. In particular, for any centrally symmetric convex body $K$ of volume $1$ in ${\mathbb R}^n$ we have that
\begin{equation*}
\frac{c_1L_K\sqrt{n}M(K)}{T_2(X_K)}\|{\bf t}\|_2\ls \mathbb{E}_{K^s}\Big\|\sum_{j=1}^s t_j x_j \Big\|_K \ls c_2T_2(X_K)\|{\bf t}\|_2.
\end{equation*}
\end{theorem}

Note that if $\vol_n(K)=1$ then $\sqrt{n}M(K)\gr c>0$, therefore the estimate is exact, up to the type-$2$ constant, and actually implies
an upper bound for $L_K$.
\end{remark}

\section{The unconditional case}

The case where $C_1,\ldots ,C_s$ and $K$ are isotropic unconditional convex bodies in ${\mathbb R}^n$ has been essentially studied in \cite[Theorem~4.1]{GHT}.

\begin{theorem}\label{theorem.GHT.unconditional}
There exists an absolute constant $c>0$ with the following property: if $K$ and $C_1,\ldots,C_s$ are isotropic unconditional convex bodies
in $\mathbb{R}^n$ then, for every $q\gr 1$,
\begin{equation*}
\Big(\int_{C_1}\ldots\int_{C_s} \Big\| \sum_{j=1}^s t_jx_j\Big\|_K^q \,dx_1\ldots dx_s \Big)^{1/q}
\ls cn^{1/q}\sqrt{q} \cdot \max\{\|{\bf t}\|_2,\sqrt{q}\|{\bf t}\|_\infty\}\ls cn^{1/q}q\,\|{\bf t}\|_2,
\end{equation*}
for every ${\bf t}=(t_1,\ldots,t_s)\in \mathbb{R}^s$. In particular,
\begin{equation*}\|{\bf t}\|_{{\cal C},K}\ls c\sqrt{\log n}\cdot \max\{\|{\bf t}\|_2,\sqrt{\log n}\|{\bf t}\|_\infty\}
\ls c\log n\,\|{\bf t}\|_2.\end{equation*}
\end{theorem}

\begin{proof}We briefly sketch the argument, which is essentially the same as in \cite{GHT}.
We write $\mu_n$ for the uniform distribution on $B_1^n$, with density $\frac{d\mu_n(x)}{dx}=\frac{n!}{2^n}{\mathbf 1}_{B_1^n}(x)$.
If we set $\Delta_n=\{x\in {\mathbb R}_+^n: x_1+\cdots +x_n\ls 1\}$ then a simple computation shows that for
every $n$-tuple of non-negative integers $p_1,\ldots ,p_n$, one has
\begin{equation*}\int_{\Delta_n}x_1^{p_1}\ldots x_n^{p_n}dx=\frac{p_1!\cdots p_n!}{(n+p_1+\cdots +p_n)!}.\end{equation*}
In \cite{BN1} it is proved that for every isotropic unconditional convex body $K$ in ${\mathbb
R}^n$ one has $cB_{\infty }^n\subseteq K\subseteq V_n$, where $V_n=\sqrt{3/2}nB_1^n$ and $c>0$ is an absolute constant.
Therefore, $\|\cdot\|_K\ls c_1\|\cdot\|_{\infty }\ls c_1\|\cdot\|_q$, where $c_1>0$ is an
absolute constant. We proceed to give an upper bound for
\begin{equation*}F_{{\cal C},q}({\bf t}):=\int_{C_1}\cdots\int_{C_s}\big\|\sum_{i=1}^st_ix_i\big\|_{2q}^{2q}\, dx_1\cdots dx_s,\end{equation*}
where $q\geq 1$ is an integer. We write $x_i=(x_{i1},\ldots ,x_{in})$ and define $y_j=(x_{1j},\ldots
,x_{sj})$ for all $j=1,\ldots ,n$. Then,
\begin{equation*}F_{{\cal C},q}({\bf t}) =\int_{C_1}\cdots\int_{C_s}\sum_{j=1}^n\langle {\bf t},y_j\rangle^{2q}
\,dx_1\cdots dx_s =\sum_{j=1}^n\sum_{q_1+\cdots +q_s=q}\frac{(2q)!}{(2q_1)!\cdots (2q_s)!}
\prod_{i=1}^st_i^{2q_i}\int_{C_i}x_{ij}^{2q_i}\,dx_i.\end{equation*}
Next, we apply a comparison theorem from \cite{BN2}: for every function
$F:{\mathbb R}^n\rightarrow {\mathbb R}$ which is centrally symmetric,
coordinatewise increasing and absolutely continuous, we have that
\begin{equation*}\int F(x)d\mu_A(x)\ls \int F(x)d\mu_{V_n}(x),\end{equation*}
where $\mu_A$ is the uniform measure on the isotropic unconditional convex body $A$. It follows that
\begin{equation*}\int_{C_i}x_{ij}^{2q_i}\,dx_i\ls \int_{V_n}x_1^{2q_i}d\mu_{V_n}(x)\ls
(c_1n)^{2q_i}n!\int_{\Delta_n}x_1^{2q_i}dx=(c_1n)^{2q_i}\frac{n!(2q_i)!}{(n+2q_i)!},\end{equation*}
where $c_1=\sqrt{3/2}$. Combining the above we see that
\begin{equation*}F_{{\cal C},q}({\bf t})\ls n (n!)^s(c_1n)^{2q}(2q)!\sum_{q_1+\cdots +q_s=q}
\frac{t_1^{2q_1}\cdots t_s^{2q_s}}{(n+2q_1)!\cdots
(n+2q_s)!}.\end{equation*} Using the estimate $(n+2r)!\geq n!n^{2r}$ which holds for every
$r\geq 0$, we get
\begin{equation*}F_{{\cal C},q}({\bf t})\ls nc_1^{2q}(2q)!\sum_{q_1+\cdots +q_s=q}t_1^{2q_1}\cdots t_s^{2q_s}.\end{equation*}
We now use another observation from \cite{BN2}: if $q\geq 1$ is an integer and $P_q(y)=\sum_{q_1+\cdots +q_s=q}y_1^{q_1}\cdots y_s^{q_s}$, $y=(y_1,\ldots ,y_s)\in {\mathbb R}_+^s$,
then for any $y\in {\mathbb R}_+^s$ with $y_1+\cdots +y_s=1$ we have
\begin{equation*}P_q(y)\ls\left (2e\max\{ 1/q,\| y\|_{\infty }\}\right )^q.\end{equation*}
Applying this inequality to the $s$-tuple $y=\frac{1}{\|{\bf t}\|_2^2}\big ( t_1^2,\ldots ,t_s^2\big )$ we get
\begin{equation*}
F_{{\cal C},q}^{\frac{1}{2q}}({\bf t}) \ls  c_1n^{\frac{1}{2q}}\sqrt[2q]{(2q)!}\left (2e\max\{\|{\bf t}\|_2^2/q,\|{\bf t}\|_{\infty }^2 \}\right )^{1/2}
\ls  c_2n^{\frac{1}{2q}}\sqrt{q}\max\{ \|{\bf t}\|_2,\sqrt{q}\|{\bf t}\|_{\infty }\}. \end{equation*}
Then, we easily conclude the proof.\end{proof}

\begin{remark}\label{rem:unco-1}\rm Using our approach we can obtain a similar upper bound directly.
Consider ${\bf t}\in {\mathbb R}^s$ with $\|{\bf t}\|_2=1$. As usual, we have
\begin{equation*}\|{\bf t}\|_{C^s,K}=L_C\, I_1(\mu_{{\bf t}},K),\end{equation*}
where $\mu_{{\bf t}}$ is an unconditional isotropic log-concave probability measure. Since $K$ is also unconditional and
isotropic, we have $c_1B_{\infty }^n\subseteq K$ and hence $\|x\|_K\ls c_1^{-1}\|x\|_{\infty }$ for all $x\in {\mathbb R}^n$.
Therefore,
\begin{equation*}I_1(\mu_{{\bf t}},K)=\int_{{\mathbb R}^n}\|x\|_Kd\mu_{{\bf t}}(x)\ls
c_1^{-1}\int_{{\mathbb R}^n}\max_{1\ls i\ls n}|\langle x,e_i\rangle |\,d\mu_{{\bf t}}(x)\ls c_2\log n
\end{equation*}because $\mu_{{\bf t}}$ is an isotropic $\psi_1$-measure with an absolute constant $\varrho $
(see \cite[Proposition~3.5.8]{AGA-book}). Since $C$ is unconditional, we also have $L_C\ls c_3$ for some
absolute constant $c_3>0$; it follows that
\begin{equation*}\|{\bf t}\|_{C^s,K}\ls c_4\log n\,\|{\bf t}\|_2\end{equation*}
for every ${\bf t}\in {\mathbb R}^s$. Of course the estimate of Theorem~\ref{theorem.GHT.unconditional} is more delicate, and
can be better by a $\sqrt{\log n}$-term, as it depends on the coordinates of ${\bf t}$.
\end{remark}

\begin{remark}\label{rem:unco-2}\rm In \cite{GHT} it is observed that the $\ell_\infty$-term in the estimate
provided by Theorem~\ref{theorem.GHT.unconditional} cannot be removed. If $C=\overline{B_1^n}$ and
$K=\frac{1}{2} B_{\infty }^n$ then choosing the vector ${\bf e_1}=(1,0,0,\ldots,0)$ we have
\begin{equation*}\|{\bf e_1}\|_{C^s,K}= \int_{\overline{B_1^n}} 2\|x\|_\infty\,dx\gr c\log n\,\|{\bf e_1}\|_{\infty }\end{equation*}
for some absolute constant $c>0$.

The example of the cube shows that the term $\sqrt{\log n}\|{\bf t}\|_2$ is necessary.
Gluskin and V.~Milman show in \cite{Gluskin-VMilman-2004} that if $C=K=\frac{1}{2}B_{\infty}^n$ then
\begin{equation*}\|{\bf t}\|_{K^n,K}\approx q_n({\bf t})=\sum_{i=1}^{u}t_i^{\ast }+\sqrt{u}\left (\sum_{i=u+1}^n
(t_i^{\ast })^2\right )^{1/2}\end{equation*} where $u\approx \log n$
and $(t_i^{\ast })_{i\ls n}$ is the decreasing rearrangement of $(|t_j|)_{j=1}^n$. It is observed in \cite[Remark~4.5]{GHT}
that this implies the lower bound
\begin{equation*}\int_{S^{n-1}}\|{\bf t}\|_{K^n,K}\,d\sigma ({\bf t})\gr c\sqrt{\log n}.\end{equation*}
\end{remark}

\begin{remark}\label{rem:ell-p}\rm It is interesting to test the results of Section~4 and Section~5 on the example of the
$\ell_p^n$-balls $B_p^n$. Let us first assume that $1\ls p\ls 2$. Then, $\ell_p^n$ has cotype-2 constant bounded by an absolute
(independent from $p$ and $n$) constant. It is also known (see \cite[Chapter~5]{AGA-book}) that $M(B_p^n)\approx n^{\frac{1}{p}-\frac{1}{2}}$ and
$\vol_n(B_p^n)^{1/n}\approx n^{-\frac{1}{p}}$. It follows that
\begin{equation*}M(\overline{B_p^n})=\vol_n(B_p^n)^{1/n}M(B_p^n)\approx 1/\sqrt{n}.\end{equation*}
Since $\overline{B_p^n}$ is isotropic and its isotropic constant is also bounded by an absolute constant, Theorem~\ref{th:third} shows that
\begin{equation*}\|{\bf t}\|_{\overline{B_p^n}^s,\overline{B_p^n}} \ls c_1\,\|{\bf t}\|_2\end{equation*}
for every $s\gr 1$ and ${\bf t}\in {\mathbb R}^s$, where $c_1>0$ is an absolute constant.

Next, let us assume that $2\ls q\ls \infty$. It is then known (see \cite[Chapter~5]{AGA-book})
that $\vol_n(B_q^n)^{1/n}\approx n^{-\frac{1}{q}}$ and
\begin{equation*}M(B_q^n)\approx \min\{\sqrt{q},\sqrt{\log n}\}n^{\frac{1}{q}-\frac{1}{2}}.\end{equation*}
It follows that
\begin{equation*}M(\overline{B_q^n})=\vol_n(B_q^n)^{1/n}M(B_q^n)\approx \min\{\sqrt{q},\sqrt{\log n}\}/\sqrt{n}.\end{equation*}
Since $\overline{B_q^n}$ is an isotropic $\psi_2$-convex body with constant $\varrho\approx 1$ (independent from $q$
and $n$ -- see \cite{Barthe-Guedon-Mendelson-Naor-2005}) and its isotropic constant is also bounded by an absolute constant,
Theorem~\ref{th:psi2-case} shows that
\begin{equation*}\|{\bf t}\|_{\overline{B_q^n}^s,\overline{B_q^n}} \ls c_2\min\{\sqrt{q},\sqrt{\log n}\}\,\|{\bf t}\|_2\end{equation*}
for every $s\gr 1$ and ${\bf t}\in {\mathbb R}^s$, where $c_2>0$ is an absolute constant.
\end{remark}

\section{Applications to vector balancing problems}

Let $\mu $ be an isotropic log-concave probability measure on ${\mathbb R}^n$ and $K$ be a centrally symmetric convex body in ${\mathbb R}^n$. Our starting observation is that
\begin{equation*}\int_{O(n)}I_1(\mu ,U(K))\,d\nu (U) =\int_{{\mathbb R}^n}\int_{O(n)}\|x\|_{U(K)}d\nu (U)\,d\mu (x)=
M(K)\int_{{\mathbb R}^n}\|x\|_2d\mu (x)\approx \sqrt{n}M(K).\end{equation*}
Applying this fact for the measure $\mu_{{\bf t}}$, from \eqref{eq:isotropic-identity} we immediately get the following.

\begin{proposition}\label{prop:rot-lower-bound}Let $C$ be an isotropic convex body in $\mathbb{R}^n$ and $K$ be a centrally symmetric convex body
in $\mathbb{R}^n$. For every ${\bf t}=(t_1,\ldots,t_s)\in \mathbb{R}^s$ there exists $U\in O(n)$ such that
\begin{equation}\label{eq:isotropic-identity-2}\|{\bf t}\|_{U(C)^s,K}\gr cL_C\sqrt{n}M(K)\,\|{\bf t}\|_2.\end{equation}
\end{proposition}

We know that if ${\rm vol}_n(K)=1$ then the quantity $\sqrt{n}M(K)$ is always greater than $c$. Therefore, Proposition \ref{prop:rot-lower-bound}
provides many examples in which the lower bound of Gluskin and V.~Milman can be improved (note also the presence of $L_C$ in the
right hand side of the inequality). For example, in the classical example of the cube $K=\dfrac{1}{2}B_{\infty }^n$ we have that
$\sqrt{n}M(K)\approx\sqrt{\log n}$, which implies the following:

\begin{corollary}For every isotropic convex body $C$ in ${\mathbb R}^n$ and any ${\bf t}=(t_1,\ldots,t_s)\in \mathbb{R}^s$ there exists $U\in O(n)$ such that
\begin{equation*}\int_{U(C)}\cdots\int_{U(C)}\Big\|\sum_{j=1}^st_jx_j\Big\|_{\infty }dx_1\cdots dx_s
\gr cL_C\sqrt{\log n}\,\|{\bf t}\|_2,\end{equation*}
where $c>0$ is an absolute constant.\end{corollary}

In this section we explore further this idea. We shall use a number of important facts from
asymptotic convex geometry (see \cite{BGVV} for proofs and additional references). For any centrally symmetric convex body $K$ in
${\mathbb R}^n$ and any $q\neq 0$ we define
\begin{equation*}M_q(K)=\left (\int_{S^{n-1}}\|\xi\|_K^qd\sigma (\xi )\right )^{1/q}.\end{equation*}
Litvak, V.~Milman and Schechtman have proved in \cite{Litvak-VMilman-Schechtman-1998} that
\begin{equation}\label{eq:Mq-M}M_q(K)\approx M(K)\end{equation}
for every $1\ls q\ls c_1k(K)$, where $c_1>0$ is an absolute constant and $k(K)=n(M(K)/b(K))^2$
is the Dvoretzky dimension of $K$. Moreover, Klartag and Vershynin have proved in \cite{Klartag-Vershynin-2007} that
\begin{equation}\label{eq:M-q-M}M_{-q}(K)\approx M(K)\end{equation}
for every $1\ls q\ls c_2d(K)$, where $d(K)\gr c_3k(K)$ is a parameter of $K$ defined by
\begin{equation*}d(K)=\min\Big\{ n,-\log\gamma_n\Big(\frac{m(K)}{2}K\Big)\Big\},\end{equation*}
and $m(K)\approx \sqrt{n}M(K)$ is the median of $\|\cdot\|_K$ with
respect to the standard Gaussian measure $\gamma_n$ on $\mathbb{R}^n$.

For any isotropic log-concave probability measure $\mu $ on $\mathbb{R}^n$ and any $q\neq 0$, $q>-n$, let
\begin{equation*}I_q(\mu ) := \Big( \int_{{\mathbb R}^n} \|x\|_2^q\,d\mu (x) \Big)^{1/q}.\end{equation*}
Paouris has proved in \cite{Paouris-1} and \cite{Paouris-2} that
\begin{equation}\label{eq:Iq}I_{-q}(\mu ) \approx I_q(\mu )\approx \sqrt{n}\end{equation}
for every $1\ls q\ls c_4q_{\ast }(\mu )$, where $q_{\ast}(\mu):=\max\{q:k(Z_q^{\circ }(\mu )\gr q\}$. It is known
that $q_{\ast }(\mu )\gr c_5\sqrt{n}$. Moreover, if $\mu $ is a $\psi_2$-measure with constant $\varrho $ then
$q_{\ast }(\mu)\gr c_6n/\varrho^2$.

\begin{theorem}\label{th:rot-lower-bound-q}Let $C$ be an isotropic centrally symmetric convex body in $\mathbb{R}^n$ and $K$ be a centrally symmetric convex body
in $\mathbb{R}^n$. Then, for every ${\bf t}=(t_1,\ldots,t_s)\in \mathbb{R}^s$ and $S\subseteq E_2^n$
with $|S|\ls e^{q({\bf t})}$, a random $U\in O(n)$ satisfies
\begin{equation*}\vol_{ns}\Big(\Big\{(x_j)_{j=1}^s:x_j\in U(C)\;\;\hbox{for all}\;j\;\hbox{and}\,
\Big\|\sum_{j=1}^s\epsilon_jt_jx_j\Big\|_K\ls cL_C\sqrt{n}M(K)\,\|{\bf t}\|_2
\;\;\hbox{for some}\;\epsilon\in S\Big\}\Big)\ls e^{-q({\bf t})}\end{equation*}
with probability greater than $1-e^{-2q({\bf t})}$, where
\begin{equation*}q({\bf t}):=\min\{q_{\ast }(\mu_{{\bf t}} ),d(K)\}.\end{equation*}
\end{theorem}

\begin{proof}We may assume that $\|{\bf t}\|_2=1$. We start by writing
\begin{equation*}\int_C\cdots\int_C\Big\|\sum_{j=1}^st_jx_j\Big\|_K^{-q({\bf t})}dx_1\cdots dx_s=\int_{{\mathbb R}^n}\|x\|_K^{-q({\bf t})}d\nu_{{\bf t}}(x)
=L_C^{-q({\bf t})}\int_{{\mathbb R}^n}\|x\|_K^{-q({\bf t})}d\mu_{{\bf t}}(x).\end{equation*}
It follows that
\begin{align*}\int_{O(n)}\int_C\cdots\int_C\Big\|\sum_{j=1}^st_jx_j\Big\|_{U(K)}^{-q({\bf t})}dx_1\cdots dx_s \,d\nu (U)
&= L_C^{-q({\bf t})}\int_{{\mathbb R}^n}\int_{O(n)}\|x\|_{U(K)}^{-q({\bf t})}d\nu (U)\,d\mu_{{\bf t}} (x)\\
&=L_C^{-q({\bf t})}M_{-q({\bf t})}^{-q({\bf t})}(K)\int_{{\mathbb R}^n}\|x\|_2^{-q({\bf t})}d\mu_{{\bf t}} (x)\\
&= L_C^{-q({\bf t})}I_{-q({\bf t})}^{-q({\bf t})}(\mu_{{\bf t}} )M_{-q({\bf t})}^{-q({\bf t})}(K).\end{align*}
From Markov's inequality, a random $U\in O(n)$ satisfies
\begin{equation*}\int_C\cdots\int_C\Big\|\sum_{j=1}^st_jx_j\Big\|_{U(K)}^{-q({\bf t})}
dx_1\cdots dx_s \ls e^{2q({\bf t})}L_C^{-q({\bf t})}I_{-q({\bf t})}^{-q({\bf t})}(\mu_{{\bf t}} )M_{-q({\bf t})}^{-q({\bf t})}(K)\end{equation*}
with probability greater than $1-e^{-2q({\bf t})}$. Since
\begin{equation*}\int_C\cdots\int_C\Big\|\sum_{j=1}^st_jx_j\Big\|_{U(K)}^{-q({\bf t})}dx_1\cdots dx_s
=\int_C\cdots\int_C\Big\|\sum_{j=1}^s\epsilon_jt_jx_j\Big\|_{U(K)}^{-q({\bf t})}dx_1\cdots dx_s\end{equation*}
for every $\epsilon\in E_2^s$, we conclude that a random $U\in O(n)$ satisfies
\begin{equation*}\int_C\cdots\int_C\Big\|\sum_{j=1}^s\epsilon_jt_jx_j\Big\|_{U(K)}^{-q({\bf t})}dx_1\cdots dx_s
\ls e^{2q({\bf t})}L_C^{-q({\bf t})}I_{-q({\bf t})}^{-q({\bf t})}(\mu_{{\bf t}} )M_{-q({\bf t})}^{-q({\bf t})}(K)\end{equation*}
for all $\epsilon\in E_2^s$, with probability greater than $1-e^{-2q({\bf t})}$.

Next, fix any such $U$ and let $S\subseteq E_2^n$ with $|S|\ls e^{q({\bf t})}$. Using \eqref{eq:M-q-M}, \eqref{eq:Iq}
and Markov's inequality, we see that a random
$s$-tuple $(x_1,\ldots ,x_s)\in C^s$ satisfies
\begin{equation*}\Big\|\sum_{j=1}^s\epsilon_jt_jx_j\Big\|_{U(K)}\gr e^{-3}L_CI_{-q({\bf t})}(\mu_{\bf t} )M_{-q({\bf t})}(K)\gr c_1L_C\sqrt{n}M(K)\end{equation*}
for all $\epsilon\in S$, with probability greater than $1-e^{-q({\bf t})}$.\end{proof}

\medskip

Recall that if $C$ is a $\psi_2$-body with constant $\varrho $ then $\mu_{{\bf t}}$ is a $\psi_2$ isotropic log-concave probability
measure with constant $\varrho $. In this case $q_{\ast }(\mu_{{\bf t}})\gr cn/\varrho^2$, and hence, in Theorem~\ref{th:rot-lower-bound-q}
we have $q({\bf t})\gr c\min\{n/\varrho^2,d(K)\}$. Moreover, if $C=\overline{B_2^n}$ we have that $U(C)=\overline{B_2^n}$ for all $U\in O(n)$
and $\varrho\approx 1$. Therefore, we have the following corollary.

\begin{corollary}\label{cor:rot-lower-bound-q}Let $C$ be an isotropic centrally symmetric convex body in $\mathbb{R}^n$ which is $\psi_2$ with
constant $\varrho $, and $K$ be a centrally symmetric convex body in $\mathbb{R}^n$. Then, for every ${\bf t}=(t_1,\ldots,t_s)\in \mathbb{R}^s$
and $S\subseteq E_2^s$ with $|S|\ls e^{c\min\{n/\varrho^2,d(K)\}}$, a random $U\in O(n)$ satisfies
\begin{align*} &\vol_{ns}\Big(\Big\{(x_j)_{j=1}^s:x_j\in U(C)\;\;\hbox{for all}\;j\;\hbox{and}\,
\Big\|\sum_{j=1}^s\epsilon_jt_jx_j\Big\|_K\ls c_1L_C\sqrt{n}M(K)\,\|{\bf t}\|_2
\;\;\hbox{for some}\;\epsilon\in S\Big\}\Big)\\
&\hspace*{1.5cm}\ls e^{-c_2\min\{n/\varrho^2,d(K)\}}\end{align*}
with probability greater than $1-e^{-c_2\min\{n/\varrho^2,d(K)\}}$. In particular, for any centrally symmetric convex body $K$ in ${\mathbb R}^n$
and any $S\subseteq E_2^s$ with $|S|\ls e^{cd(K)}$ we have
\begin{equation*}\vol_{ns}\Big(\Big\{(x_j)_{j=1}^s:x_j\in \overline{B_2^n}\;\;\hbox{for all}\;j\;\hbox{and}\,
\,\Big\|\sum_{j=1}^s\epsilon_jt_jx_j\Big\|_K\ls c_1L_C\sqrt{n}M(K)\,\|{\bf t}\|_2\;\;\hbox{for some}\;\epsilon\in S\Big\}\Big )
\ls e^{-c_2d(K)}.\end{equation*}
\end{corollary}

\begin{remark}\rm Choosing $t_1=\cdots =t_s=1$, one may view the previous results as lower bounds for a ``randomized" version of the
parameter $\beta_s(C,K)$. A general lower bound for $\beta_n(C,K)$ was proved by Banaszczyk; in \cite{Banaszczyk-1993} he showed that if $C$ and $K$ are
centrally symmetric convex bodies in ${\mathbb R}^n$ then
\begin{equation}\label{eq:ban-lower}\beta_n(C,K)\gr c\sqrt{n}(\vol_n(C)/\vol_n(K))^{1/n}\end{equation}
for an absolute constant $c>0$. An alternative proof of this lower bound can be deduced from a more general result of
Gluskin and V.~Milman in \cite{Gluskin-VMilman-2004}: If $\vol_n(K)=\vol_n(C)$ then, for any $0<u<1$ one has
\begin{equation*}
\vol_{n^2}\Big(\Big\{ (x_j)_{j=1}^n: x_j\in C\;\;\hbox{for all}\;j\;\hbox{and}\; \Big\|\sum_{j=1}^nt_jx_j\Big\|_K
\ls u\|{\bf t}\|_2\Big\}\Big) \ls u^ne^{\frac{(1-u^2)n}{2}},\end{equation*}
which implies that, for each ${\bf t}\in {\mathbb R}^n$, with probability greater than $1-e^{-n}$ with respect to $(x_1,\ldots ,x_n)$ we have
\begin{equation*}\min_{\epsilon\in E_2^n}\Big\| \sum_{j=1}^n \epsilon_jt_jx_j\Big\|_K \gr \frac{1}{10}\|{\bf t}\|_2.\end{equation*}
Banaszczyk's theorem corresponds to the case $s=n$ and ${\bf t}=(1,1,\ldots ,1)$.
Starting from this observation, the first and third authors of this article proved in \cite{Chasapis-Skarmogiannis}
several results in the spirit of Theorem~\ref{th:rot-lower-bound-q} and Corollary~\ref{cor:rot-lower-bound-q}. For example,
they showed that if $K$ is a centrally symmetric convex body in $\mathbb{R}^n$ and $S\subseteq E_2^n$ then
\begin{equation*}
\vol_{n^2}\Big(\Big\{ (x_j)_{j=1}^n\subseteq B_2^n : \Big\|\sum_{j=1}^n \epsilon_jx_j\Big\|_K \ls c\delta\sqrt{n}M(K) \hbox{, for some }
\epsilon\in S\Big\}\Big)\ls |S|\cdot\gamma_n(\delta\sqrt{n}M(K)\,K)+e^{-n}.\end{equation*}
A concrete application of this fact is that, for every $1\ls p\ls \log n$ and any $S\subseteq E_2^n$ with $|S|\ls 2^{c_pn}$,
a random $n$-tuple of points in $B_2^n$ satisfies, with probability greater than $1-e^{-n}$,
\begin{equation*}\Big\|\sum_{j=1}^n\epsilon_jx_j\Big\|_p\gr c\sqrt{p}\sqrt{n}\big(\vol_n(B_2^n)/\vol_n(B_p^n)\big )^{1/n}\end{equation*}
for all $\epsilon =(\epsilon_1,\ldots ,\epsilon_n)\in S$, while in the case $p>\log n$ one can deduce that
for any $0<\delta <1$ and $S\subseteq E_2^n$ with $|S|\ls 2^{n^{1-\delta}}$, a random $n$-tuple of points in $B_2^n$ satisfies, with
probability greater than $1-e^{-n}$,
\begin{equation*}\Big\|\sum_{j=1}^n\epsilon_jx_j\Big\|_p\gr c(\delta )\sqrt{\log n}\sqrt{n}\big(\vol_n(B_2^n)/\vol_n(B_p^n)\big )^{1/n}\end{equation*}
for all $\epsilon =(\epsilon_1,\ldots ,\epsilon_n)\in S$. We can obtain (in fact, more direct) variants and generalizations of these
bounds from Corollary~\ref{cor:rot-lower-bound-q} and the available information on $d(B_p^n)$.
\end{remark}

Following the proof of Theorem~\ref{th:rot-lower-bound-q} we can also obtain upper bounds for the $\|\cdot\|_K$-norm of signed sums of random
points from an isotropic body $C$.

\begin{theorem}\label{th:rot-upper-bound-p}Let $C$ be an isotropic centrally symmetric convex body in $\mathbb{R}^n$ and $K$ be a centrally symmetric convex body
in $\mathbb{R}^n$. Then, for every ${\bf t}=(t_1,\ldots,t_s)\in \mathbb{R}^s$ and $S\subseteq E_2^n$
with $|S|\ls e^{p({\bf t})}$, a random $U\in O(n)$ satisfies
\begin{equation*}\vol_{ns}\Big(\Big\{(x_j)_{j=1}^s:x_j\in U(C)\;\;\hbox{for all}\;j\;\hbox{and}\;
\Big\|\sum_{j=1}^s\epsilon_jt_jx_j\Big\|_K\gr cL_C\sqrt{n}M(K)\,\|{\bf t}\|_2
\;\;\hbox{for some}\;\epsilon\in S\Big\}\Big )\ls e^{-p({\bf t})}\end{equation*}
with probability greater than $1-e^{-2p({\bf t})}$, where
\begin{equation*}p({\bf t}):=\min\{q_{\ast }(\mu_{{\bf t}} ),k(K)\}.\end{equation*}
\end{theorem}

\begin{proof}We may assume that $\|{\bf t}\|_2=1$. We start by writing
\begin{equation*}\int_C\cdots\int_C\Big\|\sum_{j=1}^st_jx_j\Big\|_K^{p({\bf t})}dx_1\cdots dx_s=\int_{{\mathbb R}^n}\|x\|_K^{p({\bf t})}d\nu_{{\bf t}}(x)
=L_C^{p({\bf t})}\int_{{\mathbb R}^n}\|x\|_K^{p({\bf t})}d\mu_{{\bf t}}(x).\end{equation*}
It follows that
\begin{align*}\int_{O(n)}\int_C\cdots\int_C\Big\|\sum_{j=1}^st_jx_j\Big\|_{U(K)}^{p({\bf t})}dx_1\cdots dx_s\,d\nu (U)
&= L_C^{p({\bf t})}\int_{{\mathbb R}^n}\int_{O(n)}\|x\|_{U(K)}^{p({\bf t})}d\nu (U)\,d\mu_{{\bf t}} (x)\\
&=L_C^{p({\bf t})}M_{p({\bf t})}^{p({\bf t})}(K)\int_{{\mathbb R}^n}\|x\|_2^{p({\bf t})}d\mu_{{\bf t}} (x)\\
&= L_C^{p({\bf t})}I_{p({\bf t})}^{p({\bf t})}(\mu_{{\bf t}} )M_{p({\bf t})}^{p({\bf t})}(K).\end{align*}
Then, we proceed as in the proof of Theorem~\ref{th:rot-lower-bound-q} using Markov's inequality,
and then \eqref{eq:Mq-M} and \eqref{eq:Iq}.\end{proof}

\smallskip

We can also obtain an analogue of Corollary~\ref{cor:rot-lower-bound-q} under the assumption that $C$ is a $\psi_2$-body with constant
$\varrho $. In particular, we have:

\begin{corollary}\label{cor:rot-upper-bound-p}Let $K$ be a centrally symmetric convex body in $\mathbb{R}^n$.
Then, for every ${\bf t}=(t_1,\ldots,t_s)\in \mathbb{R}^s$ and any $S\subseteq E_2^s$ with $|S|\ls e^{ck(K)}$ we have
\begin{align*}&\vol_{ns}\Big(\Big\{(x_j)_{j=1}^s:x_j\in \overline{B_2^n}\;\;\hbox{for all}\;j\;\hbox{and}\;
\Big\|\sum_{j=1}^s\epsilon_jt_jx_j\Big\|_K\gr cL_C\,\sqrt{n}M(K)\,\|{\bf t}\|_2\;\;\hbox{for some}\;\epsilon\in S\Big\}\Big)\\
&\hspace*{1.5cm}\ls e^{-ck(K)}.\end{align*}
\end{corollary}

Finally, we briefly describe the proof of Theorem~\ref{th:r-barany-grinberg}. Recall that for any centrally symmetric convex
body $K$ in ${\mathbb R}^n$ and any $\delta\in (0,1)$ the parameter $\beta_{\delta ,s}^{(R)}(K,K)$
is defined by
\begin{equation*}\beta_{\delta ,s}^{(R)}(K,K):=\min\Big\{ r>0:{\rm vol}_{ns}\Big(\Big\{(x_j)_{j=1}^s:x_j\in K
\;\hbox{for all}\;j\;\hbox{and}\;\min_{\epsilon\in E_2^s}
\Big\|\sum_{j=1}^s\epsilon_jx_j\Big\|_K\ls r\Big\}\Big)\gr 1-\delta \Big\}.\end{equation*}

\begin{proof}[Proof of Theorem~\ref{th:r-barany-grinberg}]
Our starting point is Lemma~\ref{lem:higher-moments}; applied for the vector ${\bf 1}=(1,\ldots ,1)\in {\mathbb R}^n$, it shows that for
any centrally symmetric convex body $K$ in $\mathbb{R}^n$,
\begin{equation}\label{eq:appl-1}\Big ({\mathbb E}_{K^n}\Big\|\sum_{j=1}^nx_j\Big\|_K^q\Big )^{1/q}\ls cq\,\|{\bf 1}\|_{K^n,K},\end{equation}
where $c>0$ is an absolute constant. On the other hand, by the symmetry of $K$ we have that, for any $q\gr 1$,
\begin{equation}\label{eq:appl-2}{\mathbb E}_{K^n}\Big\|\sum_{j=1}^nx_j\Big\|_K^q\,dx_1\cdots dx_n
={\mathbb E}_{K^n}\Big ({\mathbb E}_{\epsilon }\Big\|\sum_{j=1}^n\epsilon_jx_j\Big\|_K^q\Big ).\end{equation}
Combining the above we have, in particular,
\begin{equation}\label{eq:appl-3}\Big ({\mathbb E}_{K^n}\,\min_{\epsilon\in E_2^n}\Big\|\sum_{j=1}^n\epsilon_jx_j\Big\|_K^q\Big )^{1/q}
\ls c_1q\,\|{\bf 1}\|_{K^n,K}.\end{equation}
It follows that a random $n$-tuple $(x_1,\ldots ,x_n)\in K^n$ satisfies
\begin{equation*}\min_{\epsilon\in E_2^n}\Big\|\sum_{j=1}^n\epsilon_jx_j\Big\|_K\ls c_2q\,\|{\bf 1}\|_{K^n,K}\end{equation*}
with probability greater than $1-e^{-q}$. Choosing $q=\log (2/\delta )$ we see that
\begin{equation}\label{eq:delta}\beta_{\delta ,n}^{(R)}(K,K)\ls c_2\log (2/\delta )\,\|{\bf 1}\|_{K^n,K}.\end{equation}
Inserting our upper bounds for $\|{\bf 1}\|_{K^n,K}$ into \eqref{eq:delta} we
conclude the proof. \end{proof}

\bigskip

\medskip

\noindent {\bf Acknowledgements.} The contribution of the second named author to this work was made during a visit at the
University of Missouri, Columbia; he would like to thank the Department of Mathematics for the warm hospitality. The third
named author is supported by a PhD Scholarship from the Hellenic Foundation for Research and Innovation (ELIDEK); research number 70/3/14547.

\bigskip

\footnotesize
\bibliographystyle{amsplain}

\bigskip

\bigskip

\bigskip

\noindent{\bf Keywords:} convex bodies, log-concave probability measures, weighted sums of random vectors, isotropic position.
\\
\thanks{\noindent {\bf 2010 MSC:} Primary 52A23; Secondary 46B06, 52A40, 60D05.}

\bigskip

\bigskip

\noindent \textsc{Giorgos \ Chasapis}: Department of
Mathematical Sciences, Kent State University, Kent, OH 44242, USA.

\smallskip

\noindent \textit{E-mail:} \texttt{gchasap1@kent.edu}

\bigskip

\noindent \textsc{Apostolos \ Giannopoulos}: Department of
Mathematics, National and Kapodistrian University of Athens, Panepistimioupolis 157-84,
Athens, Greece.

\smallskip

\noindent \textit{E-mail:} \texttt{apgiannop@math.uoa.gr}

\bigskip

\noindent \textsc{Nikos \ Skarmogiannis}: Department of
Mathematics, National and Kapodistrian University of Athens, Panepistimioupolis 157-84,
Athens, Greece.

\smallskip

\noindent \textit{E-mail:} \texttt{nikskar@math.uoa.gr}

\bigskip

\end{document}